%% file: main.tex
\documentclass[11pt]{article}

\title{Betti Numbers of Prodsimplicial Complexes for Directed Graphs with Applications to Word Reductions}
\author{Lina Fajardo G\'omez, Margherita Maria Ferrari, Nata\v{s}a Jonoska, Masahico Saito }
\date{}

\input{preamble_style}
\input{preamble}

\begin{document}
\maketitle
\input{abstract}
\input{introduction}
\input{homology}
\input{digraphbettis}
\input{dowgraphs}
\input{dowbettis}
\input{conclusion}
\printbibliography[heading=bibintoc]
\end{document}

%% file: preamble_style.tex
\usepackage[style=numeric,sortcites]{biblatex}
\usepackage[margin=2.5cm]{geometry}
\usepackage{float}
\usepackage{etoolbox} 

\usepackage[all]{nowidow}

\usepackage[dvipsnames]{xcolor}
\definecolor{pdflinkcolor}{rgb}{.1,.1,.6}	
\definecolor{pdfcitecolor}{rgb}{.6,.1,.1}	
\definecolor{pdfanchorcolor}{rgb}{0,1,0}	
\definecolor{pdfurlcolor}{rgb}{.1,.6,.1}	
\definecolor{pdfpagecolor}{rgb}{0,0,1}	 
\definecolor{pdffilecolor}{rgb}{1,0,0}	 
\usepackage[%
breaklinks,
colorlinks=true,
citecolor=pdfcitecolor,
linkcolor=pdflinkcolor,
anchorcolor=pdfanchorcolor,
urlcolor=pdfurlcolor,
filecolor=pdffilecolor,
pagebackref=false,
pdfauthor={Lina Fajardo G\'omez},
pdftitle={Betti Numbers of Prodsimplicial Complexes for Directed Graphs with Applications to Word Reductions}, linktoc=page]{hyperref}

%% file: preamble.tex
\usepackage{inputenc}
\usepackage{mathtools, amsmath, amsthm, amssymb, amsfonts, float, url, multicol}
\usepackage{relsize} 
\usepackage{color, enumerate, mathrsfs,enumitem}
\usepackage[dvipsnames]{xcolor}
\usepackage{graphicx,pgf,epstopdf,pgfplots}
\usepackage{tikz}
\usepackage[usenames,dvipsnames,table,xcdraw,pdf]{pstricks}

\usepackage{breqn} 
\usepackage[noabbrev,capitalize]{cleveref} 
\crefname{Thm}{theorem}{theorems}
\crefname{Prop}{proposition}{propositions}
\crefname{Lem}{lemma}{lemmas}
\crefname{Coro}{corollary}{corollaries}
\crefname{Ex}{example}{examples}
\crefname{Rmk}{remark}{remarks}
\usepackage{epsfig}
\usepackage{pst-grad} 
\usepackage{pst-plot} 
\usepackage[space]{grffile} 
\usepackage{etoolbox} 
\usepackage{refcount}
\usepackage{lscape} 
\usepackage{longtable} 
\usepackage{overpic} 
\usepackage{multirow} 
\usepackage{array}
\allowdisplaybreaks

\usepackage{adjustbox}

\newtheorem{Thm}{Theorem}[section]
\newtheorem{Coro}[Thm]{Corollary}
\newtheorem{Prop}[Thm]{Proposition}
\newtheorem{Lem}[Thm]{Lemma}

\theoremstyle{definition}
\newtheorem{Def}[Thm]{Definition}
\newtheorem{Ex}[Thm]{Example}
\newtheorem{Rmk}[Thm]{Remark}

\newcommand{\N}{\mathbb{N}}
\newcommand{\Z}{\mathbb{Z}}
\newcommand{\R}{\mathbb{R}}

\newcommand{\DOW}{\Sigma_{DOW}}
\newcommand{\SOW}{\Sigma_{SOW}}
\newcommand{\MDOW}{M^{DOW}}
\newcommand{\MSOW}{M^{SOW}}

\newcommand{\subw}{\sqsubseteq}

\DeclarePairedDelimiter\abs{\lvert}{\rvert}
\DeclareMathOperator*{\Square}{\mathlarger{\square}}

%% file: abstract.tex
\abstract{We propose custom made cell complexes, in particular prodsimplicial complexes, in order to analyze data consisting of directed graphs. These are constructed by attaching  cells that are products of simplices and are suited to study data of acyclic directed graphs, called here consistently directed graphs. We investigate possible values of the first and second Betti numbers and the types of cycles that generate nontrivial homology. We apply these tools to directed graphs associated with reductions of double occurrence words, words that are associated with DNA recombination processes in certain species of ciliates. We study the effects of word operations  on the homology for these graphs.}

%% file: introduction.tex
\section{Introduction}
Topological Data Analysis (TDA) has been extensively used in recent years in biology and other sciences. It tries to capture the underlying structure of a given data set through properties such as connectedness, circular loops and higher dimensional holes. Topological invariants can be used to detect such  ``topological signatures” of the space. Typically, a given data set to analyze is a set of points in a Euclidean space, which is a priori discrete. To capture topological shapes of such data sets, simplicial complexes are formed based on the proximity of data points.

In some biological phenomena, such as DNA recombination processes, data sets are represented by graphs, which consist of vertices and edges. For example, in \cite{hajij}, gene interaction patterns are represented by graphs where vertices represent genes and edges represent how two genes interact, such as intersecting each other. In this paper, we present a novel graph-based model for genome rearrangement in some ciliate species. Gene assembly pathways can be represented by subword pattern deletions in double occurrence words (DOWs), words where each symbol appears exactly twice \cite{dowdist}. The iterated subword deletions modeling the recombination patterns can be represented by a graph, called a {\em word graph}, whose vertices are DOWs connected by a directed edge if one word can be obtained from the other through a pattern deletion. Thus, methods for performing TDA on such directed graph data are of interest.

Another shortcoming of common TDA is that its construction of complexes is often based solely on simplicial ones. In our model of word graphs for gene assembly, it seems natural to fill in squares as well as triangles to study topological aspects of the biological process. Given such a specific biological situation, we are inspired to use more general, custom built cell complexes to study topological properties of data sets. In this paper, for the purpose of specifically applying it to our word graph model, we propose the use of a prodsimplicial complex for topological studies of digraphs.

Thus, our focus of the paper is twofold: (1) defining specific cell complexes, called prodsimplicial complexes, constructed from directed graphs by attaching products of simplices for topological analysis of graph outputs, and (2) applying prodsimplicial complex homology to study the complexity of gene assembly processes through word graphs that model DNA rearrangement in certain species of ciliates. 

The paper is organized as follows. In \Cref{chap2} we define prodsimplicial cell complexes for directed graphs. In \Cref{digraphbettis} we describe some nontrivial generators of homology groups and construct digraphs with arbitrarily large Betti numbers. \Cref{chap3} introduces DOWs, their reduction pathways and word graphs, while \Cref{dowbettis} studies the effect of word operations on word graphs and their associated topological invariants.

%% file: homology.tex
\section{Prodsimplicial Homology for Directed Graphs}
\label{chap2}
Topological data analysis is used to determine the underlying shape of the space containing a data set by studying the space's invariants. For example, the $n$-dimensional Betti number, denoted by $\beta_n$, is a topological invariant corresponding to the number of ``holes'' of dimension $n$. Instead of using simplices as cells, we define a prodismplicial complex for our goal of studying DNA recombination. In this section, we recall basic concepts from graph theory and introduce the notation that are used throughout the rest of this discussion.

\subsection{Directed Graphs with Source and Target}

\label{graphs}
\label{def:sctgt}
A directed graph (digraph for short) $G$ is a 4-tuple $(V,E,\iota,\tau)$ where $V$ is a finite set, $E\subseteq V \times V \setminus \{(x,x) \ | \ x \in V\}$ and $\iota,\tau$ are maps from $E$ to $V$. Elements of $V$ and $E$ are called vertices and edges, respectively, and we use $V(G)$ and $E(G)$ to denote the sets of vertices and edges of a given digraph $G$. The map $\iota$ indicates the {\em source} (also called {\em initial}) vertex of each edge, and similarly $\tau$ indicates the {\em target} (also called {\em terminal}) vertex. Directly from the definition, it follows that the graphs considered have no loops or parallel edges. An edge $e$ with source $u$ and target $v$ is denoted \footnote{To distinguish between ordered pairs in the Cartesian product of two sets and directed edges, we use $[u,v]$ to denote the ordered pair $(u,v)$ when it corresponds to an edge, so that $[(u_1,v_1),(u_2,v_2)]$ denotes the directed edge $(u_1,v_1)\rightarrow (u_2,v_2)$ between vertices $(u_1,v_1)$ and $(u_2,v_2)$.} with $e = [u,v]$. A {\em source of $G$} is a vertex $s$ such that $\tau(e)\neq s$ for every edge $e\in E$. Similarly, a {\em target of $G$} is a vertex $t$ such that $\iota(e)\neq t$ for every edge $e\in E$.

For ease of notation, we refer to a digraph as $G=(V,E)$, omitting the maps $\iota$ and $\tau$ whenever clear from the context. We refer the reader to \cite{diestel} for other elementary definitions in graph theory not explicitly recalled here.

\begin{Def}[weakly directed, consistently directed]
A connected digraph $G$ is said to be {\em weakly directed} if it has unique source $s$ and a unique target $t$. Moreover, if $G$ is weakly directed and it has no (directed) cycles, then $G$ is called {\em consistently directed}. In a weakly directed digraph $G$ we use $s(G)$ and $t(G)$ to denote the unique source and target vertices, respectively.
\end{Def}

\begin{Ex}

\Cref{fig:weaklyor} illustrates a digraph that is weakly directed, but not consistently directed since it has a single source (shown in green), a single target (in red) and a directed cycle.
\input{tikzpikz/weaklyor}
\end{Ex}

We use the definition of Cartesian product as in \cite{vizing}. 

\begin{Def}[Cartesian product, prime graph]
Let $G_1 = (V_1,E_1)$ and $G_2 = (V_2,E_2)$ be digraphs. The {\em Cartesian product} of $G_1$ and $G_2$ is the digraph $G =G_1 \square G_2 =  (V,E)$ where $V = V_1 \times V_2$ and $[(u_1,v_1), (u_2, v_2)] \in E$ if either $u_1 = u_2$ and $[v_1, v_2] \in E_2$ or $v_1 = v_2$ and $[u_1, u_2] \in E_1$. 

A digraph $G$ is said to be {\em prime with respect to the Cartesian product} if $G = G_1 \square G_2$ implies that either $G_1$ or $G_2$ is the trivial graph on one vertex, denoted $\Delta^0$. It is shown in \cite{feigenbaum} that the factorization of a connected directed graph into prime factors is unique with respect to the Cartesian product. 
\end{Def}

\begin{sloppypar}
Note that $G_1 \square G_2 \cong G_2 \square G_1$, as $[(u_1,v_1), (u_2, v_2)] \in E(G_1\square G_2)$ if and only if $[(v_1,u_1), (v_2, u_2)] \in E(G_2 \square G_1)$ so that the Cartesian product is commutative modulo isomorphism. Moreover, as proved in \cite{imrich}, it is also associative. As a consequence, when dealing with the Cartesian product of more than two graphs we write $G = ((\ldots((G_1 \square G_2) \square G_3) \ldots ) \square G_k)$ without including nested parentheses. Similarly, we write $V(G)=\{((v_1, \ldots, v_{k-1}), v_k) \ | \ v_i\in V(G_i)\text{ for } 1\leq i\leq k\}$ without nested parentheses.
\end{sloppypar}

The Cartesian product $G_1 \square G_2$ creates one copy (referred to as a $G_1$-layer in \cite{imrich}) of $G_1$ for every vertex in $G_2$, and vice versa.
It can be shown that the Cartesian product $G_1 \square G_2$ is consistently directed if and only if both factors are consistently directed \cite{linathesis}. By associativity, this result extends to Cartesian products of several factors. 
\subsection{Prodsimplicial Complexes for Directed Graphs} 
\label{prodcells}
Vertices and edges in a graph can be considered as 0-dimensional and 1-dimensional cells respectively, therefore a graph can be endowed with a richer structure by attaching higher dimensional cells at instances of particular subgraphs. 

In this section we present the main definitions of prodsimplicial complexes introduced in \cite{kozlov} for directed graphs. Other complexes, such as the $p$-path complexes, called $p$-path complexes, have also been considered in \cite{ppaths, preprint}. Our motivation comes from a biological process where two paths of length 2 connecting two vertices correspond to two independent pathways of DNA recombination. For such path pairs, we attach solid square faces so that the two pathways are regarded as equivalent.

\begin{Def}[simplicial digraph]
\label{def:simplicialdigraph}
The {\em $n$-dimensional simplicial digraph}, denoted by $\Delta^n$, is the digraph with vertices $V(\Delta^n) = \{v_0, v_1, \ldots, v_n\}$ and edges $E(\Delta^n) = \{[v_i,v_j] \ | \ 0 \leq i < j \leq n\}$. 
\end{Def}

Note that the source and target of the edges of $\Delta^n$ are induced by the total order on the set of vertices $V(\Delta^n)=\{v_0, v_1, \ldots, v_n\}$. It follows that $\Delta^n$ is consistently directed, with source $v_0$ and target $v_n$. Moreover, the total order on the vertices guarantees that all subgraphs of $\Delta^n$ induced by a subset of $V(\Delta^n)$ are also simplicial digraphs. Also, any two simplicial digraphs on the same number of vertices are isomorphic as directed graphs. In general, $\Delta^n$ can be obtained from $\Delta^{n-1}$ by adding the new vertex $v_n$ along with the edges $[v_i,v_n]$ for $i=0,\ldots,n-1$. 

In other works, simplicial digraphs are referred to as directed cliques \cite{directedcliques,cliquecomplex} or transitive tournaments \cite{digraphclasses}. Simplicial digraphs are prime with respect to the Cartesian product. Thanks to the unique prime decomposition in \cite{feigenbaum}, prodsimplicial cells as described below are well defined.

\begin{Def}[prodsimplicial cell]
An {\em $N$-dimensional prodsimplicial cell} $P$ is the $N$-cell that is a product of simplices $\prod_{i=1}^k \Delta_i^{n_i} = \Delta_1^{n_1} \times \cdots \times \Delta_k^{n_k}$ where $n_i>0$ for all $1\leq i \leq k$ and $N= \sum_{i=1}^kn_i$. Its 1-skeleton is the Cartesian product of simplicial digraphs. That is, a graph of the form
\begin{align*}
\displaystyle\Square\limits_{i=1}^k \Delta_i^{n_i} = \Delta_1^{n_1} \square \cdots \square \Delta_k^{n_k}.
\end{align*}
For brevity, we call $P$ a prodsimplicial $N$-cell. When $k=1$ we call $P$ an $N$-simplex and denote $P$ by $\Delta^N$.
\end{Def}

Since the Cartesian product of graphs is commutative up to isomorphism, we assume that in a prodsimplicial cell $P = \Square_{i=1}^k \Delta^{n_i}_i$ the order of the factors is such that $n_1 \geq n_2 \geq \cdots \geq n_k$. From this point on, a simplex will refer exclusively to a prodsimplicial cell with a single simplicial digraph factor unless otherwise specified. Given the correspondence between simplicial digraph factors and simplices, we abuse the notation and write $P = \Square_{i=1}^k \Delta^{n_i}_i$ instead of $\prod_{i=1}^{k}\Delta^{n_i}_i$. Also note that the indices are necessary because the dimensions of the factors in the definition of $P$ may not be distinct. However, we omit the subscripts for brevity where it causes no confusion. 

We can naturally define the geometric realization of a simplicial digraph as for standard simplices. From this geometric realization, prodsimplicial $N$-cells can be viewed as subsets of $\R^{N}$ that inherit the product topology. 

\begin{Ex}
Prodsimplicial cells include not only simplices, but also cubes (as iterated products of $\Delta^1$), and triangular prisms (as products of any cell with $\Delta^1$). \Cref{fig:prismic3ex} depicts the 1-skeletons of all prodsimplicial $3$-cells.
\begin{figure}[h]
\input{tikzpikz/prismic3ex}
\end{figure}
\end{Ex}

The notion of prodsimplicial cell was first introduced for direct products of undirected graphs by Kozlov in~\cite{kozlov}. We choose Cartesian products to obtain a complex that combines some features of both simplicial and cubical complexes while the cells remain consistently directed. Complexes comprised of prodsimplicial cells are also used in knot theory \cite{CIST,CLY}.

\begin{Def}[prodsimplicial cell orientation]
\label{def:orientation}
Let $\Delta^n$ be a simplex with vertex set $V(\Delta^n) = \{v_0, v_1, \ldots, v_n\}$ and $s(\Delta^n) = v_0$. Let $N(s(\Delta^n))$ be the set of neighbors of the source, $V(\Delta^n)\setminus s(\Delta^n)$. An {\em orientation} of $\Delta^{n}$ is an equivalence class (by even permutations) of orders of $N(s(\Delta^{n}))$. An {\em oriented simplex}, denoted $[\Delta^n]$, is a simplex along with an orientation.

\begin{sloppypar}
Let $P = \Square_{i=1}^k\Delta^{n_i}$ be a prodsimplicial cell and let $s(P)$ be the source of $P$. Let $v_{i,0} = s(\Delta^{n_i})$ and $N_i = \{(v_{1,0}, v_{2,0}, \ldots, v_{i,j}, \ldots, v_{k,0}) \ | \  v_{i,j} \in N(s(\Delta^{n_i})), 1 \leq j \leq n_i\}$. Given a total order on each $N_i$, a total order on $N_1, \ldots, N_k$ determines a total order of  $N(s(P))$. An {\em orientation} of $P$ is an equivalence class (by even permutations) of such orders of $N(s(P))$. A prodsimplicial cell $P$ along with an orientation is an {\em oriented prodsimplicial cell} and is denoted by $[P]$.
\end{sloppypar}
\end{Def}

\begin{Ex}[oriented prodsimplicial cells]
Let $\Delta^2$ be the 2-simplex with $V(\Delta^2) = \{a,b, c\}$ and $\Delta^1$ be the 1-simplex with $V(\Delta^1) = \{x,y\}$. Let $P = \Delta^2 \square \Delta^1$. Then $s(P) = (a,x)$ and  $N_1 = \{(b,x),(c,x)\}$ and $N_2 = \{(a,y)\}$. The order $N_1,N_2$ of the sets of neighbors induces the order $((b,x),(c,x),(a,y))$ on $N(s(P))$, while $N_2,N_1$ induces the order $((a,y),(b,x),(c,x))$. If $N_1$ is ordered such that $(c,x)$ is before $(b,x)$, then the order $N_2,N_1$ induces the order $((a,y),(c,x),(b,x))$. The first two orders give the prodsimplicial cell the same orientation, while the third order of the vertices defines an oppositely oriented cell. 
\end{Ex}

\begin{Def}[faces, facets, boundary set]
We say $P'$ is a {\em face} of an $N$-dimensional prodsimplicial cell $P = \Square_{i=1}^k \Delta^{n_i}$ if there exist $n_i' \leq n_i \ \forall i = 1, \ldots k$ such that $P' = \Square_{i=1}^k \Delta^{n_i'}$. The {\em facets} of $P$ are its $(N-1)$-dimesional faces. The {\em boundary set} of $P$, denoted by $\partial P$, is the collection of all its facets. 
\end{Def}

\begin{Def}[prodsimplicial cell complex]
\label{def:prodcomplex}
Given a digraph $G=(V,E)$, we can inductively construct a {\em prodsimplicial cell complex} $\Gamma$ associated with $G$, denoted by $\Gamma(G)$ using the following \emph{gluing process}:
\label{gluing}

\begin{itemize}
\item Let $\Gamma^{(0)}$ be the collection of all the vertices of $G$.
\item Let $\Gamma^{(1)} = \Gamma^{(0)} \cup E$. That is, add to $\Gamma^{(0)}$ all the edges of $G$ by attaching to the current complex all the prodsimplicial $1$-cells whose facets are vertices of $G$.
\item Let $\Gamma^{(N-1)}$ denote the complex created in the first $N-1$ steps. Let $P$ be a prodsimplicial $N$-cell. We add $P$ to $\Gamma^{(N)}$ if there exists a map $\phi: \partial P \rightarrow \Gamma^{(N-1)}$ such that $\partial P$ is homeomorphic to $\phi(\partial P)$, $\phi$ preserves the orientation of each edge, and the restriction of $\phi$ to each facet of $P$ is also a homeomorphism. 
\end{itemize}
\end{Def}

\begin{Ex}[prodsimplicial cell complex]
\Cref{fig:cx1} shows the construction of a prodsimplicial cell complex $\Gamma^{(3)}$ with cells of dimension at most 3 including seven squares and a cube (with its six faces counted among the squares). The digraph $G$ corresponds to $\Gamma^{(1)}$, obtained by collecting vertices and edges of $G$. 

\begin{figure}[ht!]
\setlength{\tabcolsep}{0.5cm}
\def\arraystretch{3}
\centering\bgroup
\begin{tabular}{rp{0.2\textwidth}rp{0.2\textwidth}}
(a)&
\adjustbox{valign=t}{\input{tikzpikz/cx0}}&
(b)&
\adjustbox{valign=t}{\input{tikzpikz/cx1}}\\
(c)&
\adjustbox{valign=t}{\input{tikzpikz/cx2}}&
(d)&
\adjustbox{valign=t}{\input{tikzpikz/cx3}}
\end{tabular}
\egroup
\caption{Faces of dimension 0 are shown in (a), of dimension 1 in (b), 2 in (c) and 3 in (d) as they are added to the complex.}
\label{fig:cx1}
\end{figure}

\end{Ex} 

\subsection{Prodsimplicial Homology Groups for Directed Graphs}
\label{prodhomgps}
In this section we describe a boundary operator for prodsimplicial cells, which allows us to compute homology groups for prodsimplicial cell complexes. Although a chain complex structure is defined on a prodsimplicial complex as a special case of CW-complexes, we present an explicit boundary operator on prodsimplicial cells using the product rule for computational purposes.

Given a prodsimplicial cell complex $\Gamma$, a {\em prodsimplicial $N$-chain group} on $\Gamma$, denoted $C_N(\Gamma)$, is the free abelian group generated by oriented $N$-dimensional prodsimplicial cells of $\Gamma$. Its typical element, a {\em prodsimplicial $N$-chain}, is a finite formal linear combination of $N$-dimensional prodsimplicial cells with integer coefficients. In the Cartesian product of a simplex and a chain, the product distributes over the sum; that is \[\Delta^{n_0} \square \left(\sum_{i=1}^k \Delta^{n_i} \right)= \sum_{i=1}^k (\Delta^{n_0} \square  \Delta^{n_i}). \]

Recall that the boundary operator for simplices is defined by
\[ \partial_n(\Delta^n) = \displaystyle \sum_{i=0}^n (-1)^i [v_0,v_1,\ldots, v_{i-1},\hat{v_i},v_{i+1},\ldots, v_n],
\]
where $\hat{v_i}$ indicates that vertex $v_i$ has been deleted from the simplex. For brevity, we use $[\hat{v_i}]$ to denote the simplex $[v_0,v_1,\ldots, v_{i-1},\hat{v_i},v_{i+1},\ldots, v_n]$ in computations.

\begin{Def}[boundary operator]\label{boundary}
Let $P= \Square_{i=1}^k \Delta^{n_i}$ be a prodsimplicial $N$-cell, where $N = \sum_{i=1}^k n_i$. 
For $1 \leq i \leq k$, we use the notation \[\langle\overline{\partial_{n_i} \Delta^{n_i}}\rangle=\Delta^{n_1} \square \Delta^{n_2}\square \cdots \square \Delta^{n_{i-1}} \square \partial_{n_i} \left(\Delta^{n_{i}}\right) \square \Delta^{n_{i+1}} \square \cdots \square \Delta^{n_k}.\] 
For $1 \leq i < j \leq k$, let \[\langle\overline{\partial_{n_i} \Delta^{n_i} \partial_{n_j}\Delta^{n_j}}\rangle= \Delta^{n_1} \square  \cdots \square \partial_{n_i} \left(\Delta^{n_{i}}\right)\square \cdots \square \partial_{n_j} \left(\Delta^{n_{j}}\right) \square \ldots \square\Delta^{n_k}\]
The $N$-dimensional {\em boundary operator} applied to $P$ is defined as
\begin{align*}
\displaystyle \partial_N \left(P \right) =  \partial_N \left(\Square_{i=1}^k \Delta^{n_i}\right)&= \sum_{i=1}^k (-1)^{\alpha(i)} \displaystyle [\overline{\partial_{n_i} \Delta^{n_i}}]
\end{align*}
where $\alpha(i) =\sum_{\ell=1}^{i-1} n_\ell$ is the sum of the dimensions of the factors preceding the $i$th factor.
\end{Def}
Since the boundary operator can only be applied to oriented chains ($[\Delta^n]$ and $[P]$), and not graphs ($\Delta^n$ and $G$), we omit the square brackets for brevity. 

Note that given a prodsimplicial cell complex $\Gamma$, $\partial_N$ defines a group homomorphism $\partial_N: C_N(\Gamma) \longrightarrow C_{N-1}(\Gamma)$. We verify that the above defined operator indeed defines a chain complex. Recall, for a simplex $\Delta^{n+1}$  with vertices $\{v_0,v_1, \ldots, v_{n+1}\}$, we have $(\partial_n \circ \partial_{n+1})(\Delta^{n+1})=0$.

\begin{Prop}[$\partial^2=0$ on prodsimplicial cells]
Let $P = \Square_{i=1}^k \Delta^{n_i}$ be a prodsimplicial $(N+1)$-cell, where $N+1 = \sum_{i=1}^k n_i$. Then $(\partial_N \circ \partial_{N+1})(P)=0$.
\end{Prop}

\begin{proof}
By Definition~\ref{boundary}, we have that
\begin{align*}
(\partial_N \circ \partial_{N+1})(P)&= (\partial_N \circ \partial_{N+1})\left(\Square_{i=1}^k \Delta^{n_i}\right)\\
=&\partial_N \left( \sum_{i=1}^k (-1)^{\alpha(i)} [\overline{\partial_{n_i} \Delta^{n_i}}] \right)\\
=& \sum_{j<i} (-1)^{\alpha(i)} (-1)^{\alpha(j)} [ \overline{\partial_{n_j}\Delta^{n_j}}\overline{\partial_{n_i} \Delta^{n_i}}] +
\sum_{j>i} (-1)^{\alpha(i)}(-1)^{\alpha(j)-1}[ \overline{\partial_{n_i}\Delta^{n_i}}\overline{\partial_{n_j} \Delta^{n_j}}]\\
+& (\partial_{n_i-1} \circ \partial_{n_i})(\Delta^{n_i}).
\end{align*}
The term $(\partial_{n_i-1} \circ \partial_{n_i})(\Delta^{n_i})$ is zero. Note that when $j>i$, the dimension of the $i$th term has been decreased by one and hence the exponent on $(-1)$ for the sum over $j>i$ is $\alpha(j)-1$.

Due to this difference of 1, the two sums (over $j<i$ and $j>i$) have the same terms with opposite signs, so the final expression equals zero.
\end{proof}

With this boundary operator we 
can compute cycle groups, boundary groups, homology groups, and Betti numbers for prodsimplicial complexes of directed graphs. 

%% file: tikzpikz/weaklyor.tex
\begin{figure}[H]
\begin{center}
\psscalebox{1.0 1.0} 
{
\begin{pspicture}(0,-1.4985577)(4.1971154,1.4985577)
\psdots[linecolor=black, dotsize=0.2](1.2985576,1.4)
\psdots[linecolor=black, dotsize=0.2](2.8985577,1.4)
\psdots[linecolor=black, dotsize=0.2](1.2985576,-0.2)
\psdots[linecolor=black, dotsize=0.2](2.8985577,-0.2)

\psline[linecolor=black,linewidth=0.04, arrowsize=0.05291667cm 2.0,arrowlength=1.4,arrowinset=0.0]{->}(0.098557666,-1.4)(1.2985576,-0.2)
\psline[linecolor=black,linewidth=0.04, arrowsize=0.05291667cm 2.0,arrowlength=1.4,arrowinset=0.0]{->}(1.2985576,-0.2)(1.2985576,1.4)
\psline[linecolor=black,linewidth=0.04, arrowsize=0.05291667cm 2.0,arrowlength=1.4,arrowinset=0.0]{->}(1.2985576,1.4)(2.8985577,1.4)
\psline[linecolor=black,linewidth=0.04, arrowsize=0.05291667cm 2.0,arrowlength=1.4,arrowinset=0.0]{->}(2.8985577,1.4)(2.8985577,-0.2)
\psline[linecolor=black,linewidth=0.04, arrowsize=0.05291667cm 2.0,arrowlength=1.4,arrowinset=0.0]{->}(2.8985577,-0.2)(1.2985576,-0.2)
\psline[linecolor=black,linewidth=0.04, arrowsize=0.05291667cm 2.0,arrowlength=1.4,arrowinset=0.0]{->}(2.8985577,-0.2)(4.0985575,-1.4)
\psdots[linecolor=OliveGreen, dotsize=0.2](0.098557666,-1.4)
\psdots[linecolor=red, dotsize=0.2](4.0985575,-1.4)
\end{pspicture}
}
\caption{A weakly directed digraph that is not consistently directed. The source is colored with green, while the target is in red.}
\label{fig:weaklyor}
\end{center}
\end{figure}

%% file: tikzpikz/prismic3ex.tex
\begin{center}
\psscalebox{1.0 1.0} 
{
\begin{pspicture}(0,-1.7985578)(7.7971153,1.2985578)

\psdots[linecolor=black, dotsize=0.2](3.098557666,0.8)
\psdots[linecolor=black, dotsize=0.2](3.89855766,0.40000004)
\psdots[linecolor=black, dotsize=0.2](4.6985576,1.2)
\psdots[linecolor=black, dotsize=0.2](3.098557666,-0.79999995)
\psdots[linecolor=black, dotsize=0.2](3.89855766,-1.1999999)
\psdots[linecolor=black, dotsize=0.2](4.6985576,-0.39999998)
 
\psdots[linecolor=black, dotsize=0.2](5.8985577,0.8)
\psdots[linecolor=black, dotsize=0.2](6.4985576,0.40000004)
\psdots[linecolor=black, dotsize=0.2](6.8985577,1.0)
\psdots[linecolor=black, dotsize=0.2](7.4985576,0.6)

\psline[linecolor=black, linewidth=0.04, arrowsize=0.05291667cm 2.0,arrowlength=1.4,arrowinset=0.0]{->}(3.098557666,0.8)(3.89855766,0.40000004)
\psline[linecolor=black, linewidth=0.04, arrowsize=0.05291667cm 2.0,arrowlength=1.4,arrowinset=0.0]{->}(3.098557666,0.8)(4.6985576,1.2)
\psline[linecolor=black, linewidth=0.04, arrowsize=0.05291667cm 2.0,arrowlength=1.4,arrowinset=0.0]{->}(3.89855766,0.40000004)(4.6985576,1.2)
\psline[linecolor=black, linewidth=0.04, arrowsize=0.05291667cm 2.0,arrowlength=1.4,arrowinset=0.0]{->}(3.098557666,-0.79999995)(3.89855766,-1.1999999)
\psline[linecolor=black, linewidth=0.04, arrowsize=0.05291667cm 2.0,arrowlength=1.4,arrowinset=0.0]{->}(3.89855766,-1.1999999)(4.6985576,-0.39999998)
\psline[linecolor=black, linewidth=0.04, arrowsize=0.05291667cm 2.0,arrowlength=1.4,arrowinset=0.0]{->}(3.098557666,-0.79999995)(4.6985576,-0.39999998)
\psline[linecolor=black, linewidth=0.04, arrowsize=0.05291667cm 2.0,arrowlength=1.4,arrowinset=0.0]{->}(3.098557666,0.8)(3.098557666,-0.79999995)
\psline[linecolor=black, linewidth=0.04, arrowsize=0.05291667cm 2.0,arrowlength=1.4,arrowinset=0.0]{->}(4.6985576,1.2)(4.6985576,-0.39999998)
\psline[linecolor=black, linewidth=0.04, arrowsize=0.05291667cm 2.0,arrowlength=1.4,arrowinset=0.0]{->}(3.89855766,0.40000004)(3.89855766,-1.1999999)

\psline[linecolor=black, linewidth=0.04, arrowsize=0.05291667cm 2.0,arrowlength=1.4,arrowinset=0.0]{->}(5.8985577,0.8)(6.8985577,1.0)
\psline[linecolor=black, linewidth=0.04, arrowsize=0.05291667cm 2.0,arrowlength=1.4,arrowinset=0.0]{->}(6.4985576,0.40000004)(7.4985576,0.6)
\psline[linecolor=black, linewidth=0.04, arrowsize=0.05291667cm 2.0,arrowlength=1.4,arrowinset=0.0]{->}(5.8985577,0.8)(6.4985576,0.40000004)
\psline[linecolor=black, linewidth=0.04, arrowsize=0.05291667cm 2.0,arrowlength=1.4,arrowinset=0.0]{->}(6.8985577,1.0)(7.4985576,0.6)

\psdots[linecolor=black, dotsize=0.2](5.8985577,-0.79999995)
\psdots[linecolor=black, dotsize=0.2](6.4985576,-1.1999999)
\psdots[linecolor=black, dotsize=0.2](6.8985577,-0.59999996)
\psdots[linecolor=black, dotsize=0.2](7.4985576,-0.99999994)
\psline[linecolor=black, linewidth=0.04, arrowsize=0.05291667cm 2.0,arrowlength=1.4,arrowinset=0.0]{->}(5.8985577,-0.79999995)(6.8985577,-0.59999996)
\psline[linecolor=black, linewidth=0.04, arrowsize=0.05291667cm 2.0,arrowlength=1.4,arrowinset=0.0]{->}(6.4985576,-1.1999999)(7.4985576,-0.99999994)
\psline[linecolor=black, linewidth=0.04, arrowsize=0.05291667cm 2.0,arrowlength=1.4,arrowinset=0.0]{->}(5.8985577,-0.79999995)(6.4985576,-1.1999999)
\psline[linecolor=black, linewidth=0.04, arrowsize=0.05291667cm 2.0,arrowlength=1.4,arrowinset=0.0]{->}(6.8985577,-0.59999996)(7.4985576,-0.99999994)
\psline[linecolor=black, linewidth=0.04, arrowsize=0.05291667cm 2.0,arrowlength=1.4,arrowinset=0.0]{->}(5.8985577,0.8)(5.8985577,-0.79999995)
\psline[linecolor=black, linewidth=0.04, arrowsize=0.05291667cm 2.0,arrowlength=1.4,arrowinset=0.0]{->}(6.4985576,0.40000004)(6.4985576,-1.1999999)
\psline[linecolor=black, linewidth=0.04, arrowsize=0.05291667cm 2.0,arrowlength=1.4,arrowinset=0.0]{->}(6.8985577,1.0)(6.8985577,-0.59999996)
\psline[linecolor=black, linewidth=0.04, arrowsize=0.05291667cm 2.0,arrowlength=1.4,arrowinset=0.0]{->}(7.4985576,0.6)(7.4985576,-0.99999994)

\psdots[linecolor=black, dotsize=0.2](-0.301442,-0.79999995)
\psdots[linecolor=black, dotsize=0.2](0.698558,-1.1999999)
\psdots[linecolor=black, dotsize=0.2](1.698558,-0.59999996)
\psdots[linecolor=black, dotsize=0.2](0.4985576,1.0)
\psline[linecolor=black, linewidth=0.04, arrowsize=0.05291667cm 2.0,arrowlength=1.4,arrowinset=0.0]{->}(-0.301442,-0.79999995)(0.698558,-1.1999999)
\psline[linecolor=black, linewidth=0.04, arrowsize=0.05291667cm 2.0,arrowlength=1.4,arrowinset=0.0]{->}(-0.301442,-0.79999995)(1.698558,-0.59999996)
\psline[linecolor=black, linewidth=0.04, arrowsize=0.05291667cm 2.0,arrowlength=1.4,arrowinset=0.0]{->}(0.698558,-1.1999999)(1.698558,-0.59999996)
\psline[linecolor=black, linewidth=0.04, arrowsize=0.05291667cm 2.0,arrowlength=1.4,arrowinset=0.0]{->}(1.698558,-0.59999996)(0.4985576,1.0)
\psline[linecolor=black, linewidth=0.04, arrowsize=0.05291667cm 2.0,arrowlength=1.4,arrowinset=0.0]{->}(0.698558,-1.1999999)(0.4985576,1.0)
\psline[linecolor=black, linewidth=0.04, arrowsize=0.05291667cm 2.0,arrowlength=1.4,arrowinset=0.0]{->}(-0.301442,-0.79999995)(0.4985576,1.0)
\rput[bl](3,-1.8097281){$\Delta^2 \square \Delta^1$}
\rput[bl](5.5,-1.8097281){$\Delta^1 \square \Delta^1 \square \Delta^1$}

\rput[bl](0.4985576,-1.8097281){$\Delta^3$}
\end{pspicture}
}
\caption{Directed graphs corresponding to the 1-skeletons of prodsimplicial $3$-cells.}
\label{fig:prismic3ex}
\end{center}

%% file: tikzpikz/cx0.tex
\psscalebox{0.6 0.6} 
{
\begin{pspicture}(0.5,-5.8)(7,1.5971155)
\psdots[linecolor=blue, dotsize=0.2](3.6985579,1.4985577)
\psdots[linecolor=blue, dotsize=0.2](2.4985578,-0.9014423)
\psdots[linecolor=blue, dotsize=0.2](4.8985577,-0.9014423)
\psdots[linecolor=blue, dotsize=0.2](7.2985578,-0.9014423)
\psdots[linecolor=blue, dotsize=0.2](0.09855774,-0.9014423)
\psdots[linecolor=blue, dotsize=0.2](2.4985578,-3.3014421)
\psdots[linecolor=blue, dotsize=0.2](4.8985577,-3.3014421)
\psdots[linecolor=blue, dotsize=0.2](7.2985578,-3.3014421)
\psdots[linecolor=blue, dotsize=0.2](0.09855774,-3.3014421)
\psdots[linecolor=blue, dotsize=0.2](3.6985579,-5.7014422)
\end{pspicture}
}

%% file: tikzpikz/cx1.tex
\psscalebox{0.6 0.6} 
{
\begin{pspicture}(0.5,-5.8)(7,1.5971155)
\psdots[linecolor=black, dotsize=0.2](3.6985579,1.4985577)
\psdots[linecolor=black, dotsize=0.2](2.4985578,-0.9014423)
\psdots[linecolor=black, dotsize=0.2](4.8985577,-0.9014423)
\psdots[linecolor=black, dotsize=0.2](7.2985578,-0.9014423)
\psdots[linecolor=black, dotsize=0.2](0.09855774,-0.9014423)
\psdots[linecolor=black, dotsize=0.2](2.4985578,-3.3014421)
\psdots[linecolor=black, dotsize=0.2](4.8985577,-3.3014421)
\psdots[linecolor=black, dotsize=0.2](7.2985578,-3.3014421)
\psdots[linecolor=black, dotsize=0.2](0.09855774,-3.3014421)
\psdots[linecolor=black, dotsize=0.2](3.6985579,-5.7014422)
\psline[linecolor=blue, linewidth=0.04, arrowsize=0.05291667cm 2.0,arrowlength=1.4,arrowinset=0.0]{->}(3.6985579,1.4985577)(3.6985579,1.4985577)(7.2985578,-0.9014423)
\psline[linecolor=blue, linewidth=0.04, arrowsize=0.05291667cm 2.0,arrowlength=1.4,arrowinset=0.0]{->}(3.6985579,1.4985577)(3.6985579,1.4985577)(4.8985577,-0.9014423)
\psline[linecolor=blue, linewidth=0.04, arrowsize=0.05291667cm 2.0,arrowlength=1.4,arrowinset=0.0]{->}(3.6985579,1.4985577)(2.4985578,-0.9014423)
\psline[linecolor=blue, linewidth=0.04, arrowsize=0.05291667cm 2.0,arrowlength=1.4,arrowinset=0.0]{->}(3.6985579,1.4985577)(0.09855774,-0.9014423)
\psline[linecolor=blue, linewidth=0.04, arrowsize=0.05291667cm 2.0,arrowlength=1.4,arrowinset=0.0]{->}(0.09855774,-0.9014423)(0.09855774,-3.3014421)
\psline[linecolor=blue, linewidth=0.04, arrowsize=0.05291667cm 2.0,arrowlength=1.4,arrowinset=0.0]{->}(2.4985578,-0.9014423)(2.4985578,-3.3014421)
\psline[linecolor=blue, linewidth=0.04, arrowsize=0.05291667cm 2.0,arrowlength=1.4,arrowinset=0.0]{->}(7.2985578,-0.9014423)(7.2985578,-3.3014421)
\psline[linecolor=blue, linewidth=0.04, arrowsize=0.05291667cm 2.0,arrowlength=1.4,arrowinset=0.0]{->}(7.2985578,-0.9014423)(0.09855774,-3.3014421)
\psline[linecolor=blue, linewidth=0.04, arrowsize=0.05291667cm 2.0,arrowlength=1.4,arrowinset=0.0]{->}(2.4985578,-0.9014423)(4.8985577,-3.3014421)
\psline[linecolor=blue, linewidth=0.04, arrowsize=0.05291667cm 2.0,arrowlength=1.4,arrowinset=0.0]{->}(4.8985577,-0.9014423)(7.2985578,-3.3014421)
\psline[linecolor=blue, linewidth=0.04, arrowsize=0.05291667cm 2.0,arrowlength=1.4,arrowinset=0.0]{->}(7.2985578,-0.9014423)(4.8985577,-3.3014421)
\psline[linecolor=blue, linewidth=0.04, arrowsize=0.05291667cm 2.0,arrowlength=1.4,arrowinset=0.0]{->}(4.8985577,-0.9014423)(2.4985578,-3.3014421)
\psline[linecolor=blue, linewidth=0.04, arrowsize=0.05291667cm 2.0,arrowlength=1.4,arrowinset=0.0]{->}(2.4985578,-3.3014421)(3.6985579,-5.7014422)
\psline[linecolor=blue, linewidth=0.04, arrowsize=0.05291667cm 2.0,arrowlength=1.4,arrowinset=0.0]{->}(4.8985577,-3.3014421)(3.6985579,-5.7014422)
\psline[linecolor=blue, linewidth=0.04, arrowsize=0.05291667cm 2.0,arrowlength=1.4,arrowinset=0.0]{->}(7.2985578,-3.3014421)(3.6985579,-5.7014422)
\psline[linecolor=blue, linewidth=0.04, arrowsize=0.05291667cm 2.0,arrowlength=1.4,arrowinset=0.0]{->}(0.09855774,-3.3014421)(2.4985578,-3.3014421)
\end{pspicture}
}

%% file: tikzpikz/cx2.tex
\psscalebox{0.6 0.6} 
{
\begin{pspicture}(0.5,-5.8)(7,1.5971155)
\definecolor{colour0}{rgb}{0.0,0.0,1.0}
\psdots[linecolor=black, dotsize=0.2](3.6985579,1.4985577)
\psdots[linecolor=black, dotsize=0.2](2.4985578,-0.9014423)
\psdots[linecolor=black, dotsize=0.2](4.8985577,-0.9014423)
\psdots[linecolor=black, dotsize=0.2](7.2985578,-0.9014423)
\psdots[linecolor=black, dotsize=0.2](0.09855774,-0.9014423)
\psdots[linecolor=black, dotsize=0.2](2.4985578,-3.3014421)
\psdots[linecolor=black, dotsize=0.2](4.8985577,-3.3014421)
\psdots[linecolor=black, dotsize=0.2](7.2985578,-3.3014421)
\psdots[linecolor=black, dotsize=0.2](0.09855774,-3.3014421)
\psdots[linecolor=black, dotsize=0.2](3.6985579,-5.7014422)
\psline[linecolor=black, linewidth=0.04, arrowsize=0.05291667cm 2.0,arrowlength=1.4,arrowinset=0.0]{->}(3.6985579,1.4985577)(3.6985579,1.4985577)(7.2985578,-0.9014423)
\psline[linecolor=black, linewidth=0.04, arrowsize=0.05291667cm 2.0,arrowlength=1.4,arrowinset=0.0]{->}(3.6985579,1.4985577)(3.6985579,1.4985577)(4.8985577,-0.9014423)
\psline[linecolor=black, linewidth=0.04, arrowsize=0.05291667cm 2.0,arrowlength=1.4,arrowinset=0.0]{->}(3.6985579,1.4985577)(2.4985578,-0.9014423)
\psline[linecolor=black, linewidth=0.04, arrowsize=0.05291667cm 2.0,arrowlength=1.4,arrowinset=0.0]{->}(3.6985579,1.4985577)(0.09855774,-0.9014423)
\psline[linecolor=black, linewidth=0.04, arrowsize=0.05291667cm 2.0,arrowlength=1.4,arrowinset=0.0]{->}(0.09855774,-0.9014423)(0.09855774,-3.3014421)
\psline[linecolor=black, linewidth=0.04, arrowsize=0.05291667cm 2.0,arrowlength=1.4,arrowinset=0.0]{->}(2.4985578,-0.9014423)(2.4985578,-3.3014421)
\psline[linecolor=black, linewidth=0.04, arrowsize=0.05291667cm 2.0,arrowlength=1.4,arrowinset=0.0]{->}(7.2985578,-0.9014423)(7.2985578,-3.3014421)
\psline[linecolor=black, linewidth=0.04, arrowsize=0.05291667cm 2.0,arrowlength=1.4,arrowinset=0.0]{->}(7.2985578,-0.9014423)(0.09855774,-3.3014421)
\psline[linecolor=black, linewidth=0.04, arrowsize=0.05291667cm 2.0,arrowlength=1.4,arrowinset=0.0]{->}(2.4985578,-0.9014423)(4.8985577,-3.3014421)
\psline[linecolor=black, linewidth=0.04, arrowsize=0.05291667cm 2.0,arrowlength=1.4,arrowinset=0.0]{->}(4.8985577,-0.9014423)(7.2985578,-3.3014421)
\psline[linecolor=black, linewidth=0.04, arrowsize=0.05291667cm 2.0,arrowlength=1.4,arrowinset=0.0]{->}(7.2985578,-0.9014423)(4.8985577,-3.3014421)
\psline[linecolor=black, linewidth=0.04, arrowsize=0.05291667cm 2.0,arrowlength=1.4,arrowinset=0.0]{->}(4.8985577,-0.9014423)(2.4985578,-3.3014421)
\psline[linecolor=black, linewidth=0.04, arrowsize=0.05291667cm 2.0,arrowlength=1.4,arrowinset=0.0]{->}(2.4985578,-3.3014421)(3.6985579,-5.7014422)
\psline[linecolor=black, linewidth=0.04, arrowsize=0.05291667cm 2.0,arrowlength=1.4,arrowinset=0.0]{->}(4.8985577,-3.3014421)(3.6985579,-5.7014422)
\psline[linecolor=black, linewidth=0.04, arrowsize=0.05291667cm 2.0,arrowlength=1.4,arrowinset=0.0]{->}(7.2985578,-3.3014421)(3.6985579,-5.7014422)
\psline[linecolor=black, linewidth=0.04, arrowsize=0.05291667cm 2.0,arrowlength=1.4,arrowinset=0.0]{->}(0.09855774,-3.3014421)(2.4985578,-3.3014421)
\pspolygon[linecolor=black, linewidth=0.04, fillstyle=solid,fillcolor=colour0, opacity=0.3](2.4985578,-0.9014423)(3.6985579,1.4985577)(7.2985578,-0.9014423)(4.8985577,-3.3014421)(4.8985577,-3.3014421)
\pspolygon[linecolor=black, linewidth=0.04, fillstyle=solid,fillcolor=colour0, opacity=0.3](4.8985577,-3.3014421)(7.2985578,-0.9014423)(7.2985578,-3.3014421)(3.6985579,-5.7014422)(3.6985579,-5.7014422)
\pspolygon[linecolor=black, linewidth=0.04, fillstyle=solid,fillcolor=colour0, opacity=0.3](2.4985578,-0.9014423)(4.8985577,-3.3014421)(3.6985579,-5.7014422)(2.4985578,-3.3014421)(2.4985578,-3.3014421)
\pspolygon[linecolor=black, linewidth=0.04, fillstyle=solid,fillcolor=colour0, opacity=0.3](3.6985579,1.4985577)(4.8985577,-0.9014423)(7.2985578,-3.3014421)(7.2985578,-0.9014423)(3.6985579,1.4985577)(3.6985579,1.4985577)
\pspolygon[linecolor=black, linewidth=0.04, fillstyle=solid,fillcolor=colour0, opacity=0.3](3.6985579,1.4985577)(4.8985577,-0.9014423)(2.4985578,-3.3014421)(2.4985578,-0.9014423)(2.4985578,-0.9014423)
\pspolygon[linecolor=black, linewidth=0.04, fillstyle=solid,fillcolor=colour0, opacity=0.3](2.4985578,-3.3014421)(4.8985577,-0.9014423)(7.2985578,-3.3014421)(3.6985579,-5.7014422)(3.6985579,-5.7014422)
\pspolygon[linecolor=black, linewidth=0.04, fillstyle=solid,fillcolor=colour0, opacity=0.3](3.6985579,1.4985577)(7.2985578,-0.9014423)(0.09855774,-3.3014421)(0.09855774,-0.9014423)(0.09855774,-0.9014423)
\end{pspicture}
}

%% file: tikzpikz/cx3.tex
\psscalebox{0.6 0.6} 
{
\begin{pspicture}(0.5,-5.8)(7,1.5971155)
\definecolor{colour0}{rgb}{0.0,0.0,1.0}
\psdots[linecolor=black, dotsize=0.2](3.6985579,1.4985577)
\psdots[linecolor=black, dotsize=0.2](2.4985578,-0.9014423)
\psdots[linecolor=black, dotsize=0.2](4.8985577,-0.9014423)
\psdots[linecolor=black, dotsize=0.2](7.2985578,-0.9014423)
\psdots[linecolor=black, dotsize=0.2](0.09855774,-0.9014423)
\psdots[linecolor=black, dotsize=0.2](2.4985578,-3.3014421)
\psdots[linecolor=black, dotsize=0.2](4.8985577,-3.3014421)
\psdots[linecolor=black, dotsize=0.2](7.2985578,-3.3014421)
\psdots[linecolor=black, dotsize=0.2](0.09855774,-3.3014421)
\psdots[linecolor=black, dotsize=0.2](3.6985579,-5.7014422)
\psline[linecolor=black, linewidth=0.04, arrowsize=0.05291667cm 2.0,arrowlength=1.4,arrowinset=0.0]{->}(3.6985579,1.4985577)(3.6985579,1.4985577)(7.2985578,-0.9014423)
\psline[linecolor=black, linewidth=0.04, arrowsize=0.05291667cm 2.0,arrowlength=1.4,arrowinset=0.0]{->}(3.6985579,1.4985577)(3.6985579,1.4985577)(4.8985577,-0.9014423)
\psline[linecolor=black, linewidth=0.04, arrowsize=0.05291667cm 2.0,arrowlength=1.4,arrowinset=0.0]{->}(3.6985579,1.4985577)(2.4985578,-0.9014423)
\psline[linecolor=black, linewidth=0.04, arrowsize=0.05291667cm 2.0,arrowlength=1.4,arrowinset=0.0]{->}(3.6985579,1.4985577)(0.09855774,-0.9014423)
\psline[linecolor=black, linewidth=0.04, arrowsize=0.05291667cm 2.0,arrowlength=1.4,arrowinset=0.0]{->}(0.09855774,-0.9014423)(0.09855774,-3.3014421)
\psline[linecolor=black, linewidth=0.04, arrowsize=0.05291667cm 2.0,arrowlength=1.4,arrowinset=0.0]{->}(2.4985578,-0.9014423)(2.4985578,-3.3014421)
\psline[linecolor=black, linewidth=0.04, arrowsize=0.05291667cm 2.0,arrowlength=1.4,arrowinset=0.0]{->}(7.2985578,-0.9014423)(7.2985578,-3.3014421)
\psline[linecolor=black, linewidth=0.04, arrowsize=0.05291667cm 2.0,arrowlength=1.4,arrowinset=0.0]{->}(7.2985578,-0.9014423)(0.09855774,-3.3014421)
\psline[linecolor=black, linewidth=0.04, arrowsize=0.05291667cm 2.0,arrowlength=1.4,arrowinset=0.0]{->}(2.4985578,-0.9014423)(4.8985577,-3.3014421)
\psline[linecolor=black, linewidth=0.04, arrowsize=0.05291667cm 2.0,arrowlength=1.4,arrowinset=0.0]{->}(4.8985577,-0.9014423)(7.2985578,-3.3014421)
\psline[linecolor=black, linewidth=0.04, arrowsize=0.05291667cm 2.0,arrowlength=1.4,arrowinset=0.0]{->}(7.2985578,-0.9014423)(4.8985577,-3.3014421)
\psline[linecolor=black, linewidth=0.04, arrowsize=0.05291667cm 2.0,arrowlength=1.4,arrowinset=0.0]{->}(4.8985577,-0.9014423)(2.4985578,-3.3014421)
\psline[linecolor=black, linewidth=0.04, arrowsize=0.05291667cm 2.0,arrowlength=1.4,arrowinset=0.0]{->}(2.4985578,-3.3014421)(3.6985579,-5.7014422)
\psline[linecolor=black, linewidth=0.04, arrowsize=0.05291667cm 2.0,arrowlength=1.4,arrowinset=0.0]{->}(4.8985577,-3.3014421)(3.6985579,-5.7014422)
\psline[linecolor=black, linewidth=0.04, arrowsize=0.05291667cm 2.0,arrowlength=1.4,arrowinset=0.0]{->}(7.2985578,-3.3014421)(3.6985579,-5.7014422)
\psline[linecolor=black, linewidth=0.04, arrowsize=0.05291667cm 2.0,arrowlength=1.4,arrowinset=0.0]{->}(0.09855774,-3.3014421)(2.4985578,-3.3014421)
\pspolygon[linecolor=black, linewidth=0.04, fillstyle=solid,fillcolor=colour0, opacity=0.3](2.4985578,-0.9014423)(3.6985579,1.4985577)(7.2985578,-0.9014423)(4.8985577,-3.3014421)(4.8985577,-3.3014421)
\pspolygon[linecolor=black, linewidth=0.04, fillstyle=solid,fillcolor=colour0, opacity=0.3](4.8985577,-3.3014421)(7.2985578,-0.9014423)(7.2985578,-3.3014421)(3.6985579,-5.7014422)(3.6985579,-5.7014422)
\pspolygon[linecolor=black, linewidth=0.04, fillstyle=solid,fillcolor=colour0, opacity=0.3](2.4985578,-0.9014423)(4.8985577,-3.3014421)(3.6985579,-5.7014422)(2.4985578,-3.3014421)(2.4985578,-3.3014421)
\pspolygon[linecolor=black, linewidth=0.04, fillstyle=solid,fillcolor=colour0, opacity=0.3](3.6985579,1.4985577)(4.8985577,-0.9014423)(7.2985578,-3.3014421)(7.2985578,-0.9014423)(3.6985579,1.4985577)(3.6985579,1.4985577)
\pspolygon[linecolor=black, linewidth=0.04, fillstyle=solid,fillcolor=colour0, opacity=0.3](3.6985579,1.4985577)(4.8985577,-0.9014423)(2.4985578,-3.3014421)(2.4985578,-0.9014423)(2.4985578,-0.9014423)
\pspolygon[linecolor=black, linewidth=0.04, fillstyle=solid,fillcolor=colour0, opacity=0.6](2.4985578,-3.3014421)(4.8985577,-0.9014423)(7.2985578,-3.3014421)(3.6985579,-5.7014422)(3.6985579,-5.7014422)
\end{pspicture}
}

%% file: digraphbettis.tex
\section{Betti Numbers and Generators  for
Consistently Directed Graphs}
\label{digraphbettis}

We examine generators for the first and second homology groups and investigate possible values of Betti numbers in prodsimplicial complexes of directed graphs. We adopt the notation $\beta_n(\Gamma(G))$ to denote the $n$th Betti number of the prodsimplicial complex associated with a digraph $G$, where $\Gamma(G)$ is built through the gluing process defined in \Cref{gluing}. For brevity, we abuse the notation and write $\beta_n(G)$ instead.

As is well known, $\beta_0$ indicates the number of connected components. Since we study consistently directed digraphs, which consist of a single connected component, we take $\beta_0(G) = 1$ for all digraphs of interest.

By definition, consistently directed graphs admit no cycles in the graph theoretic sense. In the work that follows, where it causes no confusion, we use the term $n$-cycles to refer to elements of cycle groups. 

\subsection{Generators of the First Homology Groups}
\Cref{fig:cd3vs} shows two connected digraphs on three vertices. The one in  \Cref{fig:cd3vs}(a) is a closed path of the form $(v_0,v_1, v_2=v_0)$ and forms a 1-cycle, while simplicial digraphs, as shown in \Cref{fig:cd3vs}(b) do not contribute to $\beta_1$. Although closed paths of the form $(v_0,v_1, \ldots, v_n=v_0)$ contribute to $\beta_1$, they are absent in consistently directed graphs, so we exclude them in the search for homology group generators. 
\input{tikzpikz/cd3vs}

The only two consistently directed graphs on four vertices are shown in \Cref{fig:cd4vs}. One is the union of an edge and a path of length three, as depicted in \Cref{fig:cd4vs}(a), which results in a complex homeomorphic to the circle $S^1$. Another is a square cell of the form $\Delta^1 \square \Delta^1$, shown in \Cref{fig:cd4vs}(b). We summarize our observations as follows.
\input{tikzpikz/cd4vs}

\begin{Lem}
\label{badsquare}
Let $G = (V,E)$ where $V = \{v_0, v_1, v_2, v_3\}$ and $E = \{[v_0,v_3], [v_0,v_1],[v_1,v_2],[v_2,v_3] \}$, as in \Cref{fig:cd4vs}(a). Then $\beta_1(G) = 1$ and $\beta_n(G) = 0$ for all $n > 2$.
\end{Lem}

We can build graphs with arbitrarily large $\beta_1$ values by attaching edge-disjoint paths of length 3 from source to target as follows. 

\begin{Lem}
\label{Lem:multiloops}
For all $k\geq 0$, there exists a graph $G$ such that $\beta_1(G) = k$ and $\beta_n(G) = 0$ for $n> 2$.
\end{Lem}
\input{tikzpikz/multiloops}

\begin{proof}
Let $G$ be the graph in \Cref{fig:multiloops}. Then, as a topological space, $G$ is homotopic equivalent to the bouquet of $k$ circles, hence $\beta_1(G) = k$ and $\beta_n(G) = 0$ for $n> 2$.
\end{proof}

\begin{Lem}
\label{prop:cycles}
For consistently directed graphs, all 1-cycles that represent nontrivial elements of the first homology group consist of paths with a common source and target $p_1 = (s=u_0,u_1,\ldots, u_m=t)$ and $p_2 = s=v_0, v_1, \ldots, v_n=t)$ where $\{u_1, u_2, \ldots, u_{m-1}\} \cap \{v_1, v_2, \ldots, v_{n-1}\} = \varnothing$.
\end{Lem}

\subsection{Generators of the Second Homology Groups}
\label{generalH2}
This section will focus on directed graphs with vertices labeled $v_i$, so we reserve the notation $[v_0, v_1, \ldots, v_n]$ for simplices and label other low-dimensional prodsimplicial cells as sequences of vertices. For example, $[v_0,v_1,v_2,v_3]$ denotes a tetrahedron while $v_0v_1v_2v_3$ denotes a square. 

In dimension 2, the problem of finding all graphs whose corresponding prodsimplicial complex yields nontrivial 2-cycles is less straightforward. We present here a few of the simplest examples.

To consider generators of the second homology groups, we consider consistently directed graphs of at least five vertices, as Betti numbers for all consistently directed graphs on three and four vertices as shown in \Cref{fig:multiloops,fig:cd4vs}. It can be checked by inspection that the only connected and consistently directed graphs on five edges and five vertices are pairs of paths as in \Cref{prop:cycles}. For an example with six edges and five vertices with $\beta_2(G) = 1$, we consider the graph consisting of three squares as shown in \Cref{fig:3squareslabeled}. 

\begin{Lem}
\label{lem:banana}
Let $G$ be as shown in \Cref{fig:3squareslabeled}. Then $\beta_1(G) = 0$, $\beta_2(G) = 1$ and $\beta_n(G) = 0$ for all $n > 2$. 
\input{tikzpikz/3squareslabeled}
\end{Lem}

\begin{proof}
We describe each chain group explicitly in terms of the vertices of $G$. Note that squares are labeled forming a cycle, following the source vertex with its lowest lexicographic order neighbor. The chain groups, $C_n$, in the prodsimplicial complex associated with the graph are described by generators as follows:
\begin{align*}
C_0 &= \langle [v_0], [v_1], [v_2], [v_3], [v_4] \rangle\\
C_1 &= \langle [v_0,v_1],[v_0,v_2],[v_0,v_3],[v_1,v_4],[v_2,v_4],[v_3,v_4]\rangle\\
C_2 &= \langle v_0v_1v_4v_2, v_0v_2v_4v_3, v_0v_1v_4v_3\rangle.
\end{align*}
In the graph any two paths of length 2 form a consistently directed square, so that the resulting prodsimplicial complex is homeomorphic to a sphere. It follows that $\beta_1(G) =0$, $\beta_2(G) = 1$ and $\beta_n(G) = 0$ for all $n > 2$ as desired.
\end{proof}
In addition to this, we present another graph on five vertices that has nontrivial second homology group that also yields a 2-cycle. 
\begin{Lem}
\label{lem:tennis}
Let $G$ be the graph in \Cref{fig:tennisballs}(a). Then $\beta_1(G) = 0$, $\beta_2(G) = 1$ and $\beta_n(G) = 0$ for all $n > 2$.
\end{Lem}

\begin{proof}
The chain groups in the prodsimplicial complex associated with the graph are as follows:
\begin{align*}
C_0 &= \langle [v_0], [v_1], [v_2], [v_3], [v_4], [v_5] \rangle\\
C_1 &= \langle [v_0,v_1],[v_0,v_2],[v_1,v_3],[v_1,v_4],[v_2,v_3],[v_2,v_4], [v_3, v_5], [v_4,v_5]\rangle\\
C_2 &= \langle v_0v_1v_3v_2, v_0v_1v_4v_2, v_1v_3v_5v_4, v_2v_3v_5v_4\rangle.
\end{align*}

The four squares are connected in such a way that they form a complex homeomorphic to a sphere, from which the result follows.
\end{proof}
\begin{Rmk}
It is possible to add edges connecting opposite vertices of a consistently directed square (dividing it into two simplicial digraphs) in either of the graphs depicted in    \Cref{fig:3squareslabeled} or \Cref{fig:tennisballs}. This yields graphs with complexes homeomorphic to a sphere, with different types of  nontrivial polyhedral 2-cycles. For example, the graph in \Cref{fig:tennisballs}(a) can be modified to obtain the graph in \Cref{fig:tennisballs}(b). 

\begin{figure}[h]
\begin{center}
\begin{tabular}[t]{lclc}
(a) & 
\adjustbox{valign=t}{\input{tikzpikz/tennisball}}& 
(b) &
\adjustbox{valign=t}{\input{tikzpikz/alsotennisball}}
\end{tabular}
\end{center}
\caption{Graphs with nontrivial 2-cycles.}
\label{fig:tennisballs}
\end{figure}
\end{Rmk}

\subsection{Realizability of Betti Number Combinations}
Using some of the graphs from previous results, we construct graphs with larger order having specific homology groups. To prove this result we use the following lemma that follows from the Mayer-Vietoris Sequence \cite{hatcher}.

\begin{Lem}
\label{lem:mayervietoris}
Let $G$ and $H$ be consistently directed graphs such that $G \cup H$ is consistently directed, $\Gamma(G\cup H) = \Gamma(G) \cup \Gamma(H)$,  $H_0(\Gamma(G\cap H)) = \Z$, and $H_n(\Gamma(G \cap H)) = 0$ for $n\geq 1$. Then \[H_n(\Gamma(G \cup H)) \cong H_n(\Gamma(G)) \oplus H_n(\Gamma(H))\] for $n\geq 1$. 
\end{Lem}

As a direct result of this, we are able to ``glue'' together an arbitrary number of generator graphs to attain any combination of Betti numbers. 
\begin{Coro}
\label{lem:allb2}

There exists a graph $G$ such that $\beta_1(G) = 0$, $\beta_2(G) = k$ and $\beta_n(G) = 0$ for all $n > 2$.
\end{Coro}
\input{tikzpikz/multinana}

\begin{proof}
Let $G$ be as in \Cref{fig:multinana}. Then, as a topological space, $G$ is equivalent to a wedge of spheres that are glued along an edge (where homology groups are trivial). As a result, we have $\beta_1(G) = 0$, $\beta_2(G) = k$ and $\beta_n(G) = 0$ for all $n > 2$.
\end{proof}

\begin{Coro}
\label{coro:digraphs}
Let $G$ be as in \Cref{Lem:multiloops} with vertex set $\{v_0, v_1, \ldots, v_{2k+1}\}$ and $\beta_1(G) = k$, and let $H$ be as in \Cref{lem:allb2} with vertex set $\{u_0, u_1, \ldots, u_{3\ell+1}\}$ and $\beta_2(H) = \ell$ so that $v_0 = u_{3k+1}$ and $G$ and $H$ only intersect at the vertex $v_0$.

Let $G \cup H = (V(G) \cup V(H), E(G) \cup E(H))$. Then $G \cup H$ is consistently directed and satisfies
\begin{align*}
\beta_1(G \cup H) &= k\\
\beta_2(G \cup H) &= \ell\\
\beta_n(G \cup H) &= 0 \ \text{ for } n \geq 3.
\end{align*}
\end{Coro}

\begin{proof}
It can be readily checked that the unique source and target of $G \cup H$ are, respectively, $u_0$ and $v_{2k+1}$, so that $G \cup H$ is consistently directed. By assumption, $H_0(\Gamma(G\cap H)) = \Z$ and $H_n(\Gamma(G \cap H)) = 0$ for $n\geq 1$ so that the result follows from \Cref{lem:mayervietoris}.
\end{proof}
It is therefore possible to construct consistently directed graphs with first and second homology groups isomorphic to $\Z^k$ and $\Z^\ell$, respectively, for any positive integers $k$ and $\ell$. This construction is not minimal on the number of vertices. For example, the graph shown in \Cref{fig:lantern} follows a construction similar to that in \Cref{Lem:multiloops} and achieves many of the same $k$ values using fewer vertices and edges. 

\input{tikzpikz/lantern}

\begin{Lem}
\label{lem:lantern}
Let $G$ be a digraph with vertices $V(G) = \{v_0, v_1, \ldots, v_{k+1}\}$ and edges defined by the collection of paths $(v_0,v_i,v_{k+1})$ for $0 < i < k+1$. Then $\beta_1(G) = 0$, $\beta_2(G) = {k-1 \choose 2}$, and $\beta_n(G) = 0$ for all $n > 2$.
\end{Lem}

%% file: tikzpikz/cd3vs.tex
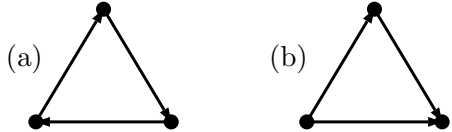
\begin{figure}[h]
\begin{center}
\psscalebox{1.0 1.0} 
{
\begin{pspicture}(0,-2.55)(5.597115,-0.85288453)
\psdots[linecolor=black, dotsize=0.2](0.09855774,-2.4514422)
\psdots[linecolor=black, dotsize=0.2](0.99855775,-0.95144224)
\psdots[linecolor=black, dotsize=0.2](1.8985578,-2.4514422)
\rput[bl](-0.3,-1.8){(a)}
\psdots[linecolor=black, dotsize=0.2](3.6985579,-2.4514422)
\psdots[linecolor=black, dotsize=0.2](4.598558,-0.95144224)
\psdots[linecolor=black, dotsize=0.2](5.4985576,-2.4514422)
\rput[bl](3.2,-1.8){(b)}
\psline[linecolor=black, linewidth=0.04, arrowsize=0.05291667cm 2.0,arrowlength=1.4,arrowinset=0.0]{->}(0.09855774,-2.4514422)(0.99855775,-0.95144224)
\psline[linecolor=black, linewidth=0.04, arrowsize=0.05291667cm 2.0,arrowlength=1.4,arrowinset=0.0]{->}(0.99855775,-0.95144224)(0.99855775,-0.95144224)(1.8985578,-2.4514422)
\psline[linecolor=black, linewidth=0.04, arrowsize=0.05291667cm 2.0,arrowlength=1.4,arrowinset=0.0]{->}(1.8985578,-2.4514422)(0.09855774,-2.4514422)
\psline[linecolor=black, linewidth=0.04, arrowsize=0.05291667cm 2.0,arrowlength=1.4,arrowinset=0.0]{->}(3.6985579,-2.4514422)(4.598558,-0.95144224)
\psline[linecolor=black, linewidth=0.04, arrowsize=0.05291667cm 2.0,arrowlength=1.4,arrowinset=0.0]{->}(4.598558,-0.95144224)(5.4985576,-2.4514422)
\psline[linecolor=black, linewidth=0.04, arrowsize=0.05291667cm 2.0,arrowlength=1.4,arrowinset=0.0]{->}(3.6985579,-2.4514422)(5.4985576,-2.4514422)
\end{pspicture}
}
\end{center}
\caption{The only two possible connected digraphs on three vertices.}
\label{fig:cd3vs}
\end{figure}

%% file: tikzpikz/cd4vs.tex
\begin{figure}[h]
\begin{center}
\psscalebox{1.0 1.0} 
{
\begin{pspicture}(0,-2.55)(4.9971156,-0.85288453)
\psdots[linecolor=black, dotsize=0.2](0.09855774,-2.4514422)
\psdots[linecolor=black, dotsize=0.2](0.09855774,-0.95144224)
\psdots[linecolor=black, dotsize=0.2](1.5985577,-0.95144224)
\psdots[linecolor=black, dotsize=0.2](1.5985577,-2.4514422)
\psdots[linecolor=black, dotsize=0.2](3.3985577,-2.4514422)
\psdots[linecolor=black, dotsize=0.2](3.3985577,-0.95144224)
\psdots[linecolor=black, dotsize=0.2](4.8985577,-0.95144224)
\psdots[linecolor=black, dotsize=0.2](4.8985577,-2.4514422)
\rput[bl](-0.7,-1.8){(a)}
\psline[linecolor=black, linewidth=0.04, arrowsize=0.05291667cm 2.0,arrowlength=1.4,arrowinset=0.0]{->}(0.09855774,-2.4514422)(0.09855774,-0.95144224)
\psline[linecolor=black, linewidth=0.04, arrowsize=0.05291667cm 2.0,arrowlength=1.4,arrowinset=0.0]{->}(0.09855774,-0.95144224)(1.5985577,-0.95144224)
\psline[linecolor=black, linewidth=0.04, arrowsize=0.05291667cm 2.0,arrowlength=1.4,arrowinset=0.0]{->}(1.5985577,-0.95144224)(1.5985577,-2.4514422)
\psline[linecolor=black, linewidth=0.04, arrowsize=0.05291667cm 2.0,arrowlength=1.4,arrowinset=0.0]{->}(0.09855774,-2.4514422)(1.5985577,-2.4514422)
\rput[bl](2.5,-1.8){(b)}
\psline[linecolor=black, linewidth=0.04, arrowsize=0.05291667cm 2.0,arrowlength=1.4,arrowinset=0.0]{->}(3.3985577,-2.4514422)(3.3985577,-0.95144224)
\psline[linecolor=black, linewidth=0.04, arrowsize=0.05291667cm 2.0,arrowlength=1.4,arrowinset=0.0]{->}(3.3985577,-0.95144224)(4.8985577,-0.95144224)
\psline[linecolor=black, linewidth=0.04, arrowsize=0.05291667cm 2.0,arrowlength=1.4,arrowinset=0.0]{<-}(4.8985577,-0.95144224)(4.8985577,-2.4514422)
\psline[linecolor=black, linewidth=0.04, arrowsize=0.05291667cm 2.0,arrowlength=1.4,arrowinset=0.0]{->}(3.3985577,-2.4514422)(4.8985577,-2.4514422)
\end{pspicture}
}
\caption{All consistenty directed graphs on four vertices.}
\label{fig:cd4vs}
\end{center}
\end{figure}
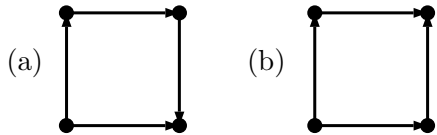

%% file: tikzpikz/multiloops.tex
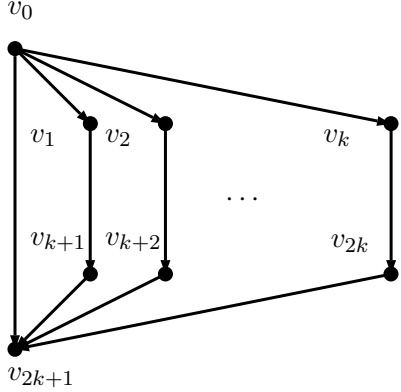
\begin{figure}[H]
\begin{center}
\psscalebox{1.0 1.0} 
{
\begin{pspicture}(0,-4.96)(5.198558,0.28)
\psdots[linecolor=black, dotsize=0.2](0.1,-0.38)
\psdots[linecolor=black, dotsize=0.2](0.1,-4.38)
\psdots[linecolor=black, dotsize=0.2](1.1,-1.38)
\psdots[linecolor=black, dotsize=0.2](1.1,-3.38)
\psline[linecolor=black, linewidth=0.04, arrowsize=0.05291667cm 2.0,arrowlength=1.4,arrowinset=0.0]{->}(0.1,-0.48)(0.1,-4.38)
\psline[linecolor=black, linewidth=0.04, arrowsize=0.05291667cm 2.0,arrowlength=1.4,arrowinset=0.0]{->}(0.1,-0.38)(1.1,-1.38)
\psline[linecolor=black, linewidth=0.04, arrowsize=0.05291667cm 2.0,arrowlength=1.4,arrowinset=0.0]{->}(1.1,-1.38)(1.1,-3.38)
\psline[linecolor=black, linewidth=0.04, arrowsize=0.05291667cm 2.0,arrowlength=1.4,arrowinset=0.0]{->}(1.1,-3.38)(0.1,-4.38)
\psdots[linecolor=black, dotsize=0.2](2.1,-1.38)
\psdots[linecolor=black, dotsize=0.2](2.1,-3.38)
\psline[linecolor=black, linewidth=0.04, arrowsize=0.05291667cm 2.0,arrowlength=1.4,arrowinset=0.0]{->}(0.1,-0.38)(2.1,-1.38)
\psline[linecolor=black, linewidth=0.04, arrowsize=0.05291667cm 2.0,arrowlength=1.4,arrowinset=0.0]{->}(2.1,-1.38)(2.1,-3.38)
\psline[linecolor=black, linewidth=0.04, arrowsize=0.05291667cm 2.0,arrowlength=1.4,arrowinset=0.0]{->}(2.1,-3.38)(0.1,-4.38)
\psdots[linecolor=black, dotsize=0.2](5.1,-1.38)
\psdots[linecolor=black, dotsize=0.2](5.1,-3.38)
\psline[linecolor=black, linewidth=0.04, arrowsize=0.05291667cm 2.0,arrowlength=1.4,arrowinset=0.0]{->}(0.1,-0.38)(5.1,-1.38)
\psline[linecolor=black, linewidth=0.04, arrowsize=0.05291667cm 2.0,arrowlength=1.4,arrowinset=0.0]{->}(5.1,-1.38)(5.1,-3.38)
\psline[linecolor=black, linewidth=0.04, arrowsize=0.05291667cm 2.0,arrowlength=1.4,arrowinset=0.0]{->}(5.1,-3.38)(0.1,-4.38)
\rput[bl](0.0,0.02){$v_0$}
\rput[bl](0.0,-4.88){$v_{2k+1}$}
\rput[bl](0.3,-1.68){$v_1$}
\rput[bl](1.3,-1.68){$v_2$}
\rput[bl](4.2,-1.68){$v_k$}
\rput[bl](2.9,-2.48){$\cdots$}
\rput[bl](0.3,-3.08){$v_{k+1}$}
\rput[bl](1.3,-3.08){$v_{k+2}$}
\rput[bl](4.3,-3.08){$v_{2k}$}
\end{pspicture}
}
\caption{$k$ disjoint paths of length 3 running ``parallel'' to $[v_0,v_{2k+1}]$.}
\label{fig:multiloops}
\end{center}
\end{figure}

%% file: tikzpikz/3squareslabeled.tex
\begin{figure}
\begin{center}
\psscalebox{1.0 1.0} 
{
\begin{pspicture}(0,-3.36)(3.2985578,0.28)
\psline[linecolor=black, linewidth=0.04, arrowsize=0.05291667cm 2.0,arrowlength=1.4,arrowinset=0.0]{->}(1.9139731,-0.22262017)(3.1725855,-1.5094359)
\psline[linecolor=black, linewidth=0.04, arrowsize=0.05291667cm 2.0,arrowlength=1.4,arrowinset=0.0]{->}(3.1725855,-1.5094359)(1.8857698,-2.7680483)
\psline[linecolor=black, linewidth=0.04, arrowsize=0.05291667cm 2.0,arrowlength=1.4,arrowinset=0.0]{->}(0.6271574,-1.4812326)(1.8857698,-2.7680483)
\psline[linecolor=black, linewidth=0.04, arrowsize=0.05291667cm 2.0,arrowlength=1.4,arrowinset=0.0]{->}(1.9139731,-0.22262017)(0.6271574,-1.4812326)
\psline[linecolor=black, linewidth=0.04, arrowsize=0.05291667cm 2.0,arrowlength=1.4,arrowinset=0.0]{->}(1.9139731,-0.22262017)(1.8998715,-1.4953343)
\psline[linecolor=black, linewidth=0.04, arrowsize=0.05291667cm 2.0,arrowlength=1.4,arrowinset=0.0]{->}(1.8998715,-1.4953343)(1.8857698,-2.7680483)
\rput[bl](1.5,0.02){$v_0$}
\rput[bl](0.0,-1.48){$v_1$}
\rput[bl](1.2,-1.48){$v_2$}
\rput[bl](2.4,-1.48){$v_3$}
\rput[bl](1.5,-3.28){$v_4$}
\psdots[linecolor=black, dotsize=0.2](1.9,-2.78)
\psdots[linecolor=black, dotsize=0.2](3.2,-1.48)
\psdots[linecolor=black, dotsize=0.2](1.9,-0.18)
\psdots[linecolor=black, dotsize=0.2](1.9,-1.48)
\psdots[linecolor=black, dotsize=0.2](0.6,-1.48)
\end{pspicture}
}
\caption{Consistently directed graph whose prodsimplicial complex consists of three squares, six edges and five vertices.}
\label{fig:3squareslabeled}

\end{center}
\end{figure}
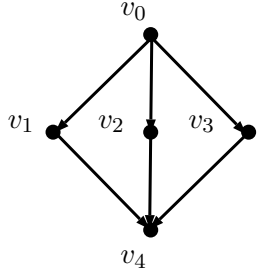

%% file: tikzpikz/tennisball.tex
\psscalebox{1.0 1.0} 
{
\begin{pspicture}(0,-6.01)(3.5026953,1.13)
\definecolor{colour0}{rgb}{0.8,0.6,0.8}
\psdots[linecolor=black, dotsize=0.2](2.8,0.57)
\psdots[linecolor=black, dotsize=0.2](0.8,-1.43)
\psdots[linecolor=black, dotsize=0.2](2.8,-1.43)
\psdots[linecolor=black, dotsize=0.2](2.8,-3.43)
\psdots[linecolor=black, dotsize=0.2](0.8,-3.43)
\psdots[linecolor=black, dotsize=0.2](2.8,-5.43)
\psline[linecolor=black, linewidth=0.04, fillstyle=solid,fillcolor=colour0, opacity=0.4, arrowsize=0.05291667cm 2.0,arrowlength=1.4,arrowinset=0.0]{->}(2.8,0.57)(2.8,-1.43)
\psline[linecolor=black, linewidth=0.04, fillstyle=solid,fillcolor=colour0, opacity=0.4, arrowsize=0.05291667cm 2.0,arrowlength=1.4,arrowinset=0.0]{->}(2.8,0.57)(0.8,-1.43)
\psline[linecolor=black, linewidth=0.04, fillstyle=solid,fillcolor=colour0, opacity=0.4, arrowsize=0.05291667cm 2.0,arrowlength=1.4,arrowinset=0.0]{->}(0.8,-1.43)(2.8,-3.43)
\psline[linecolor=black, linewidth=0.04, fillstyle=solid,fillcolor=colour0, opacity=0.4, arrowsize=0.05291667cm 2.0,arrowlength=1.4,arrowinset=0.0]{->}(2.8,-1.43)(2.8,-3.43)
\psline[linecolor=black, linewidth=0.04, fillstyle=solid,fillcolor=colour0, opacity=0.4, arrowsize=0.05291667cm 2.0,arrowlength=1.4,arrowinset=0.0]{->}(2.8,-1.43)(0.8,-3.43)
\psline[linecolor=black, linewidth=0.04, fillstyle=solid,fillcolor=colour0, opacity=0.4, arrowsize=0.05291667cm 2.0,arrowlength=1.4,arrowinset=0.0]{->}(0.8,-1.43)(0.8,-3.43)
\psline[linecolor=black, linewidth=0.04, fillstyle=solid,fillcolor=colour0, opacity=0.4, arrowsize=0.05291667cm 2.0,arrowlength=1.4,arrowinset=0.0]{->}(0.8,-3.43)(2.8,-5.43)
\psline[linecolor=black, linewidth=0.04, fillstyle=solid,fillcolor=colour0, opacity=0.4, arrowsize=0.05291667cm 2.0,arrowlength=1.4,arrowinset=0.0]{->}(2.8,-3.43)(2.8,-5.43)
\rput[bl](2.5,0.87){$v_0$}
\rput[bl](0.0,-1.33){$v_1$}
\rput[bl](2.9,-1.33){$v_2$}
\rput[bl](0.0,-3.43){$v_3$}
\rput[bl](2.9,-3.33){$v_4$}
\rput[bl](2.3,-5.93){$v_5$}
\end{pspicture}
}

%% file: tikzpikz/alsotennisball.tex
\psscalebox{1.0 1.0} 
{
\begin{pspicture}(0,-6.01)(3.5026953,1.13)
\definecolor{colour0}{rgb}{0.0,0.0,1.0}
\psdots[linecolor=black, dotsize=0.2](2.8,0.57) 
\psdots[linecolor=black, dotsize=0.2](0.8,-1.43) 
\psdots[linecolor=black, dotsize=0.2](2.8,-1.43)
\psdots[linecolor=black, dotsize=0.2](2.8,-3.43)
\psdots[linecolor=black, dotsize=0.2](0.8,-3.43) 
\psdots[linecolor=black, dotsize=0.2](2.8,-5.43)
\pspolygon[linecolor=black, linewidth=0.04, fillstyle=solid,fillcolor=colour0, opacity=0.3](2.8,0.57)(0.8,-1.43)(0.8,-3.43)(2.8,-1.43)
\psline[linecolor=black, linewidth=0.04, fillstyle=solid,fillcolor=colour0, opacity=0.4, arrowsize=0.05291667cm 2.0,arrowlength=1.4,arrowinset=0.0]{->}(2.8,0.57)(0.8,-3.43)
\psline[linecolor=black, linewidth=0.04, fillstyle=solid,fillcolor=colour0, opacity=0.4, arrowsize=0.05291667cm 2.0,arrowlength=1.4,arrowinset=0.0]{->}(2.8,0.57)(2.8,-1.43)
\psline[linecolor=black, linewidth=0.04, fillstyle=solid,fillcolor=colour0, opacity=0.4, arrowsize=0.05291667cm 2.0,arrowlength=1.4,arrowinset=0.0]{->}(2.8,0.57)(0.8,-1.43)
\psline[linecolor=black, linewidth=0.04, fillstyle=solid,fillcolor=colour0, opacity=0.4, arrowsize=0.05291667cm 2.0,arrowlength=1.4,arrowinset=0.0]{->}(0.8,-1.43)(2.8,-3.43)
\psline[linecolor=black, linewidth=0.04, fillstyle=solid,fillcolor=colour0, opacity=0.4, arrowsize=0.05291667cm 2.0,arrowlength=1.4,arrowinset=0.0]{->}(2.8,-1.43)(2.8,-3.43)
\psline[linecolor=black, linewidth=0.04, fillstyle=solid,fillcolor=colour0, opacity=0.4, arrowsize=0.05291667cm 2.0,arrowlength=1.4,arrowinset=0.0]{->}(2.8,-1.43)(0.8,-3.43)
\psline[linecolor=black, linewidth=0.04, fillstyle=solid,fillcolor=colour0, opacity=0.4, arrowsize=0.05291667cm 2.0,arrowlength=1.4,arrowinset=0.0]{->}(0.8,-1.43)(0.8,-3.43)
\psline[linecolor=black, linewidth=0.04, fillstyle=solid,fillcolor=colour0, opacity=0.4, arrowsize=0.05291667cm 2.0,arrowlength=1.4,arrowinset=0.0]{->}(0.8,-3.43)(2.8,-5.43)
\psline[linecolor=black, linewidth=0.04, fillstyle=solid,fillcolor=colour0, opacity=0.4, arrowsize=0.05291667cm 2.0,arrowlength=1.4,arrowinset=0.0]{->}(2.8,-3.43)(2.8,-5.43)
\rput[bl](2.5,0.87){$v_0$}
\rput[bl](0.0,-1.33){$v_1$}
\rput[bl](2.9,-1.33){$v_2$}
\rput[bl](0.0,-3.43){$v_3$}
\rput[bl](2.9,-3.33){$v_4$}
\rput[bl](2.3,-5.93){$v_5$}
\end{pspicture}
}

%% file: tikzpikz/multinana.tex
\begin{figure}[h]
\begin{center}
\psscalebox{1.0 1.0} 
{
\begin{pspicture}(0,-6.51)(4.537031,3.43)
\psdots[linecolor=black, dotsize=0.2](3.3,2.87)
\psdots[linecolor=black, dotsize=0.2](3.3,1.37)
\psdots[linecolor=black, dotsize=0.2](2.1,1.37)
\psdots[linecolor=black, dotsize=0.2](0.9,1.37)
\psdots[linecolor=black, dotsize=0.2](0.9,-0.13)
\psdots[linecolor=black, dotsize=0.2](3.3,-0.13)
\psdots[linecolor=black, dotsize=0.2](2.1,-0.13)
\psdots[linecolor=black, dotsize=0.2](0.9,-1.63)
\psline[linecolor=black, linewidth=0.04, arrowsize=0.05291667cm 2.0,arrowlength=1.4,arrowinset=0.0]{->}(3.3,2.87)(3.3,1.37)
\psline[linecolor=black, linewidth=0.04, arrowsize=0.05291667cm 2.0,arrowlength=1.4,arrowinset=0.0]{->}(3.3,2.87)(2.1,1.37)
\psline[linecolor=black, linewidth=0.04, arrowsize=0.05291667cm 2.0,arrowlength=1.4,arrowinset=0.0]{->}(3.3,2.87)(0.9,1.37)
\psline[linecolor=black, linewidth=0.04, arrowsize=0.05291667cm 2.0,arrowlength=1.4,arrowinset=0.0]{->}(0.9,1.37)(0.9,-0.13)
\psline[linecolor=black, linewidth=0.04, arrowsize=0.05291667cm 2.0,arrowlength=1.4,arrowinset=0.0]{->}(2.1,1.37)(0.9,-0.13)
\psline[linecolor=black, linewidth=0.04, arrowsize=0.05291667cm 2.0,arrowlength=1.4,arrowinset=0.0]{->}(3.3,1.37)(0.9,-0.13)
\psline[linecolor=black, linewidth=0.04, arrowsize=0.05291667cm 2.0,arrowlength=1.4,arrowinset=0.0]{->}(3.3,1.37)(3.3,-0.13)
\psline[linecolor=black, linewidth=0.04, arrowsize=0.05291667cm 2.0,arrowlength=1.4,arrowinset=0.0]{->}(3.3,1.37)(2.1,-0.13)
\psline[linecolor=black, linewidth=0.04, arrowsize=0.05291667cm 2.0,arrowlength=1.4,arrowinset=0.0]{->}(0.9,-0.13)(0.9,-1.63)
\psline[linecolor=black, linewidth=0.04, arrowsize=0.05291667cm 2.0,arrowlength=1.4,arrowinset=0.0]{->}(2.1,-0.13)(0.9,-1.63)
\psline[linecolor=black, linewidth=0.04, arrowsize=0.05291667cm 2.0,arrowlength=1.4,arrowinset=0.0]{->}(3.3,-0.13)(0.9,-1.63)
\psdots[linecolor=black, dotsize=0.2](0.9,-4.33)
\psdots[linecolor=black, dotsize=0.2](3.3,-4.33)
\psdots[linecolor=black, dotsize=0.2](2.1,-4.33)
\psdots[linecolor=black, dotsize=0.2](0.9,-5.83)
\psline[linecolor=black, linewidth=0.04, arrowsize=0.05291667cm 2.0,arrowlength=1.4,arrowinset=0.0]{->}(3.3,-2.83)(0.9,-4.33)
\psline[linecolor=black, linewidth=0.04, arrowsize=0.05291667cm 2.0,arrowlength=1.4,arrowinset=0.0]{->}(3.3,-2.83)(3.3,-4.33)
\psline[linecolor=black, linewidth=0.04, arrowsize=0.05291667cm 2.0,arrowlength=1.4,arrowinset=0.0]{->}(3.3,-2.83)(2.1,-4.33)
\psline[linecolor=black, linewidth=0.04, arrowsize=0.05291667cm 2.0,arrowlength=1.4,arrowinset=0.0]{->}(0.9,-4.33)(0.9,-5.83)
\psline[linecolor=black, linewidth=0.04, arrowsize=0.05291667cm 2.0,arrowlength=1.4,arrowinset=0.0]{->}(2.1,-4.33)(0.9,-5.83)
\psline[linecolor=black, linewidth=0.04, arrowsize=0.05291667cm 2.0,arrowlength=1.4,arrowinset=0.0]{->}(3.3,-4.33)(0.9,-5.83)
\psdots[linecolor=black, dotsize=0.2](3.3,-2.83)
\rput[bl](3.0,3.17){$v_0$}
\rput[bl](0.0,1.27){$v_1$}
\rput[bl](2.4,1.27){$v_2$}
\rput[bl](3.6,1.27){$v_3$}
\rput[bl](0.0,-0.23){$v_4$}
\rput[bl](2.4,-0.23){$v_5$}
\rput[bl](3.6,-0.23){$v_6$}
\rput[bl](0.0,-1.73){$v_7$}
\rput[bl](0.0,-4.23){$v_{3k-2}$}
\rput[bl](3.3,-2.53){$v_{3(k-1)}$}
\rput[bl](2.3,-4.33){$v_{3k-1}$}
\rput[bl](3.6,-4.33){$v_{3k}$}
\rput[bl](0.6,-6.53){$v_{3k+1}$}
\rput[bl](2,-2.23){$\vdots$}
\end{pspicture}
}
\caption{Graph whose corresponding prodsimplicial complex has second homology of rank $k$.}
\label{fig:multinana}
\end{center}
\end{figure}
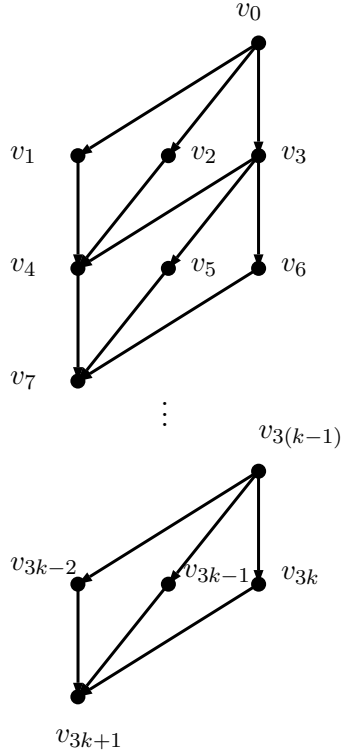

%% file: tikzpikz/lantern.tex
\begin{figure}[h]
\centering\psscalebox{1.0 1.0} 
{
\begin{pspicture}(0,-3.81)(6.592617,0.13)
\psdots[linecolor=black, dotsize=0.2](1.5,-0.43)
\psdots[linecolor=black, dotsize=0.2](0.6,-1.93)
\psdots[linecolor=black, dotsize=0.2](2.4,-1.93)
\psdots[linecolor=black, dotsize=0.2](1.5,-3.43)
\psdots[linecolor=black, dotsize=0.2](3.3,-1.93)
\psdots[linecolor=black, dotsize=0.2](5.4,-1.93)
\psline[linecolor=black, linewidth=0.04, arrowsize=0.05291667cm 2.0,arrowlength=1.4,arrowinset=0.0]{->}(1.5,-0.43)(2.4,-1.93)
\psline[linecolor=black, linewidth=0.04, arrowsize=0.05291667cm 2.0,arrowlength=1.4,arrowinset=0.0]{->}(1.5,-0.43)(0.6,-1.93)
\psline[linecolor=black, linewidth=0.04, arrowsize=0.05291667cm 2.0,arrowlength=1.4,arrowinset=0.0]{->}(0.6,-1.93)(1.5,-3.43)
\psline[linecolor=black, linewidth=0.04, arrowsize=0.05291667cm 2.0,arrowlength=1.4,arrowinset=0.0]{->}(2.4,-1.93)(1.5,-3.43)
\psline[linecolor=black, linewidth=0.04, arrowsize=0.05291667cm 2.0,arrowlength=1.4,arrowinset=0.0]{->}(1.5,-0.43)(3.3,-1.93)
\psline[linecolor=black, linewidth=0.04, arrowsize=0.05291667cm 2.0,arrowlength=1.4,arrowinset=0.0]{->}(3.3,-1.93)(1.5,-3.43)
\psline[linecolor=black, linewidth=0.04, arrowsize=0.05291667cm 2.0,arrowlength=1.4,arrowinset=0.0]{->}(1.5,-0.43)(5.4,-1.93)
\psline[linecolor=black, linewidth=0.04, arrowsize=0.05291667cm 2.0,arrowlength=1.4,arrowinset=0.0]{->}(5.4,-1.93)(1.5,-3.43)
\rput[bl](1.2,-0.13){$v_0$}
\rput[bl](1.5,-3.93){$v_{k+1}$}
\rput[bl](0.0,-1.93){$v_1$}
\rput[bl](1.7,-1.93){$v_2$}
\rput[bl](3.6,-1.93){$v_3$}
\rput[bl](5.6,-1.93){$v_k$}
\end{pspicture}
}
\caption{Directed graph with ${k-1 \choose 2}$ 2-cycles.}

\label{fig:lantern}
\end{figure}

%% file: dowgraphs.tex
\section{Word Graphs of Double Occurrence Words}
\label{chap3}

In this section we focus on specific biomolecular processes where consistently directed graphs appear and the prodsimplicial complexes can be applied. 

Massive rearrangement processes are observed during the development of somatic nuclei in certain species of ciliates such as \textit{Oxytricha trifallax's}. The recombination is guided by short DNA repeats flanking the DNA segments that are rearranged and guiding their order. These short repeats can be modeled  by a sequence of double occurrence words, words where each symbol appears twice \cite{landweberref,landweberref2,prestonref}. In particular, over 90\% of DNA rearrangement in these species can be described through an iterated process of deletion of repeat and return words \cite{patterns}. Repeat and return words are generalizations of square and palindromic factors in words and are of interest in language theory \cite{patterns, dowdist, ryanthesis}. In this section we describe graphs associated with double occurrence words that model these DNA rearrangement process
\cite{assgraphs2}.

\subsection{Double Occurrence Words}
We call an ordered, countable set of symbols $\Sigma$ an {\em alphabet}. A {\em word over $\Sigma$} is a finite sequence of the form $w=a_1a_2\ldots a_n$ where $a_i \in \Sigma$ whose {\em length}, denoted $|w|$, is  $n$. We denote with $\Sigma^*$ the set of all words over $\Sigma$, including the empty word, denoted by $\epsilon$. The set of all symbols comprising a word $w$ is denoted by $\Sigma[w]$. The {\em reverse} of a word $w = a_1 a_2 \ldots a_n$ is $w^R=a_n a_{n-1} \ldots a_2 a_1$. The word $v$ is a {\em factor} of the word $w$, denoted $v \subw w$, if there exist $w_1, w_2 \in \Sigma^\ast$ such that $w = w_1vw_2$. In this presentation we set $\Sigma \subseteq[n]$ for some $n\in \N$. For example $w = 122313$ is a word over $[3]=\{1,2,3\}$ of length $\abs{w} = 6$. The reverse of $w = 122313$ is the word $w^R = 313221$.

A word $w \in \Sigma^*$ is called a {\em double occurrence word} (DOW) if every symbol in $\Sigma$ appears in $w$ either zero or two times. We use $\Sigma_{DOW}$ to denote the set of all DOWs over $\Sigma$. Similarly, we call $w$ a {\em single occurrence word} (SOW) if each symbol in $\Sigma$ appears either once or not at all. The set of SOWs over $\Sigma$ is denoted by $\Sigma_{SOW}$. Since a DOW of length $n$ uses $n/2$ distinct symbols, we say that the {\em size} of a DOW $w$ is $|w|/2$. When we restrict $\Sigma_{DOW}$ to DOWs of size less than or equal to $n$, we denote the set by $\Sigma_{DOW}^{\leq n}$.

A word $w \in \Sigma^*$ is said to be in {\em ascending order} if $a_1 = \min(\Sigma[w])$ and the first appearance of each symbol is the immediate successor of the largest of all the preceding symbols. For example, the word $w_1=122313$ is a DOW in ascending order, while the DOW $w_2=133212$ is not. 

We say that $w_1$ and $w_2$ are {\em ascending order equivalent}, and write $w_1 \sim w_2$, if there exists a bijection on $\Sigma$ inducing a morphism $f$ on $\Sigma^\ast$ such that $f(w_1) = w_2$. Words $w_1=122313$ and $w_2=133212$ are equivalent via the bijective map given by: $ 1\mapsto 1,\; 3\mapsto 2,\;
2\mapsto 3. $ Since words in ascending order are unique, up to this equivalence we consider words in ascending order as representatives of the classes determined by the relation $\sim$. 

From now on, we consider only equivalence classes of DOWs and abuse the notation by writing words in place of their equivalence class where no confusion arises. The following definition can be found in \cite{digon}.

\begin{Def}[repeat word, return word]\label{def:Mw}
Let  $x,y,z \in \Sigma^*$ and $u \in (\Sigma \setminus \Sigma[w])_{SOW}$. We say that
\begin{itemize}
\item the word $uu$ is a {\em repeat word} in $w = xuyuz$ and the word $xyz$ is obtained from $w$ by a {\em repeat deletion} denoted $d_u(w) = xyz$. In this case we call $u$ a {\em repeat factor} in $w$.
\item the word $uu^R$ is a {\em return word} in $ w = xuy u^Rz$ and the word $xyz$ is obtained from $w$ by a {\em return deletion}, also denoted $d_u(w) = xyz$. In this case we call $u$ a {\em return factor} in $w$.
\end{itemize}
Repeat or return words, $uu$ or $uu^R$, where $\abs{u}=1$ are called {\em trivial}. We say that a word $uu$ (resp. $uu^R$) is a {\em maximal} repeat (resp. return) word in $w$ if there are no other repeat (resp. return) factors $v$ in $w$ containing $u$ with $\abs{v}>\abs{u}$. Following \cite{ryan}, we use $\MDOW_w$ to denote the set of maximal repeat ($uu$) or return ($uu^R$) words in $w$. In addition, we define the set of repeat or return factors in $w$ as $\MSOW_w = \{u \subw w \ | \ uu \in \MDOW_w \text{ or } uu^R \in \MDOW_w\}$.
\end{Def}

\begin{Lem}\cite{ryan}
Let $w$ be a DOW of size $n$. For each $x\in\Sigma[w]$ there exists a unique $u\in \MSOW_w$ such that $x\in\Sigma[u]$.
\end{Lem}

The set of maximal repeat or return words may include trivial repeat or return words, as illustrated in \Cref{ex:deletion}. In ascending order, some of the factors may be equivalent so elements of $\MDOW_w$ (resp. $\MSOW_w$) are always written as DOWs (resp. SOWs) and not their equivalence classes, so that every symbol in $w$ appears in some word in $\MSOW_w$. We are interested in maximal repeat and return words, and present the following definition, which was adapted from the so-called ``pattern reduction'' process described in  \cite{ryan} and \cite{dowdist}. 

\begin{Def}[successor, predecessor]
\begin{sloppypar}
The set $D(w) = \bigcup_{u \in \MSOW_w} \{v \ | \ v \text{ is in ascending order and }v \sim d_u(w)\}$ is called the {\em set of immediate successors of $w$}. If there exists a sequence of words $w=w_1, w_2, \ldots, w_n = w'$ such that $w_i\sim d_{u_{i}}(w_{i-1})$ for some choice of $u_i \in \MSOW_{w_i}$, we call $w'$ a {\em successor} of $w$ and $w$ a {\em predecessor} of $w'$. Note that the empty word $\epsilon$ is a successor of all words.
\end{sloppypar}
\end{Def}

\begin{Ex}\label{ex:deletion}
\begin{sloppypar}
Let $w= 1234523541$. The set of maximal repeat or return words in $w$ is $\MDOW_w = \{ 11,2323, 4554\}$ and $\MSOW_w = \{1, 23, 45\}$. Since $d_{1}(w) = 23452354\sim 12341243$ and $d_{23}(w) = 145541\sim 123321$, we have that the set of immediate successors of $w$ is $D(w) = \{12341243, 123321, 123231\}$. We may continue to delete subwords from the successors as follows. The maximal repeat or return factors of $12341243, 123321,$ and $123231$ are
\end{sloppypar}
\begin{align*}
\MSOW_{12341243} &= \{12, 34\}\\
\MSOW_{123321} &= \{123\}\\
\MSOW_{123231} &= \{1, 23\}
\end{align*}
so that their deletions yield

\begin{minipage}{\linewidth}
\begin{minipage}[t]{0.3\linewidth}
\begin{align*}
d_{12}(12341243) &=3443\sim 1221\\
d_{34}(12341243) &=1212\\
D(12341243) &= \{1221, 1212\}
\end{align*}
\end{minipage}
\begin{minipage}[t]{0.3\linewidth}
\begin{align*}
d_{123}(123321) &= \epsilon\\
D(123321) &= \{\epsilon\}
\end{align*}
\end{minipage}
\begin{minipage}[t]{0.3\linewidth}
\begin{align*}
d_{1}(123231) & =2323\sim 1212\\
d_{23}(123231) &= 11\\
D(123231) &= \{1212,11\}.
\end{align*}
\end{minipage}
\end{minipage}

Repeating this process once more yields $D(1212) = D(1221) = D(11) = \{\epsilon\}$. We can therefore say that the set of all successors of $w$ is the set
\[\{12341243, 123321, 123231, 1221, 1212, 11, \epsilon\}.\]
\end{Ex}

\subsection{Word Graphs}
We refer the reader to \cite{diestel} for elementary definitions in graph theory and to \Cref{def:sctgt} for the definitions of a directed graph, source, and target. 

\begin{Def}[global word graph]
The {\em global word graph} $G_n=(V,E)$ {\em of double occurrence words of size $n$} is the graph defined by:

\begin{itemize}
\item $V(G_n)=\DOW^{\leq n}/_\sim$;
\item $E(G_n) = \bigcup_{w\in V} E_w$, where $E_w = \{[w,v
]\ | v \in D(w)\}$. 
\end{itemize}

For a vertex $w$ in $G_n$, we define the {\em word graph rooted at $w$}, denoted $G_w$, as the induced subgraph of the global word graph containing as vertices $w$ and all of its successors. 
\end{Def}
By construction $G_w$ does not contain any cycles, and has unique source $w$ and unique target $\epsilon$ hence word graphs are consistently directed.

The figures in this section were computer generated. Though we do not include the characters in our presentation, the labels on the vertices are separated by commas to improve readability. 
\begin{Ex}
The global word graph $G_2$ of size $2$ is shown in \Cref{fig:2graph}, with the word graph rooted at $1122$ highlighted in blue. The vertex set is $V(G_{2}) = \DOW^{\leq 2}/_\sim = \{\epsilon, 11, 1221, 1212, 1122\}$. 
The word graph rooted at $w=1234523541$ whose successors are computed in \Cref{ex:deletion} is depicted in \Cref{fig:1234523541}. 
\begin{figure}[h]
\centering\includegraphics[width=0.5\textwidth]{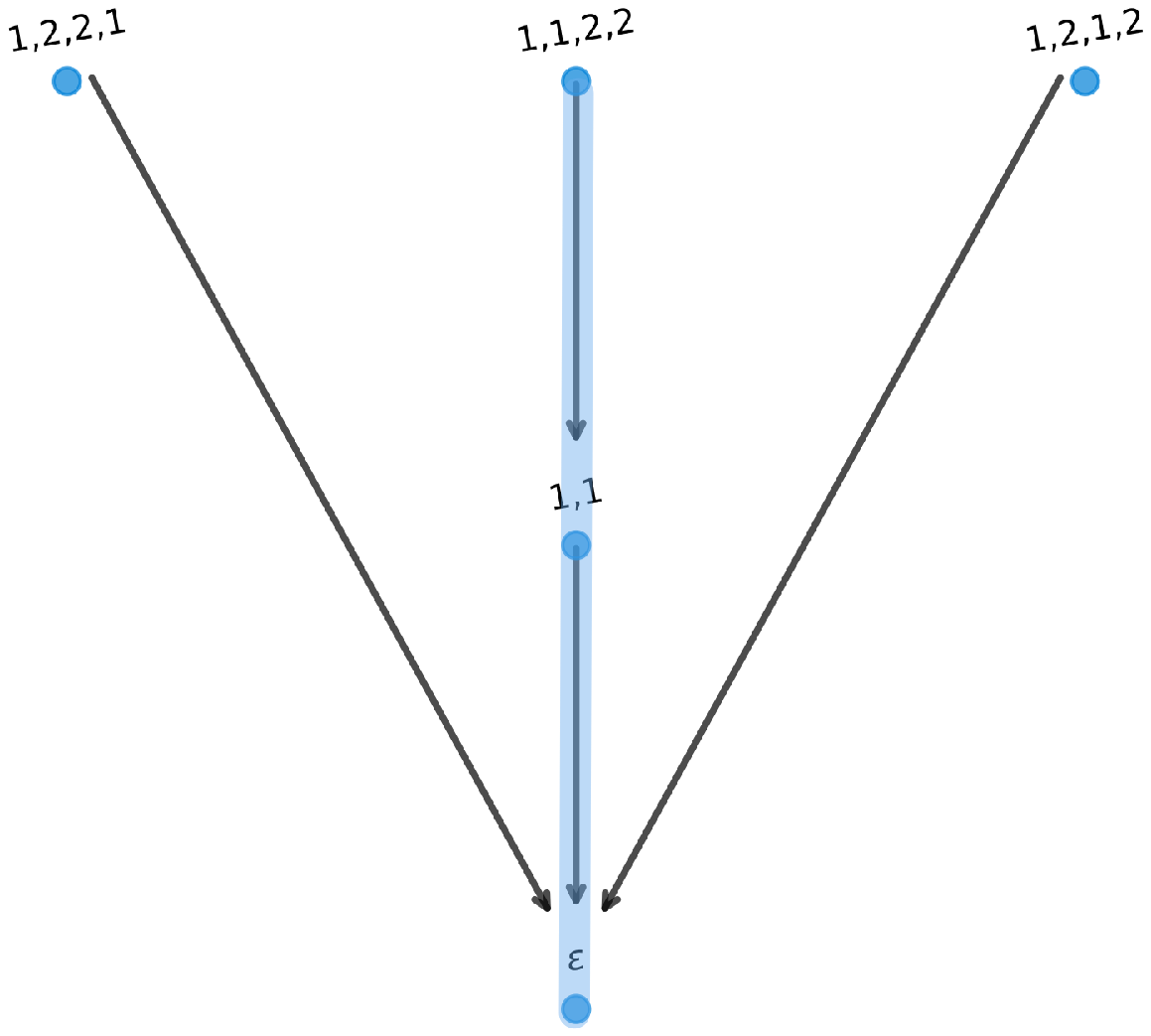}
\caption{Global word graph of size $2$.}
\label{fig:2graph}
\end{figure}

\begin{figure}[h]
\centering\includegraphics[width=0.5\textwidth]{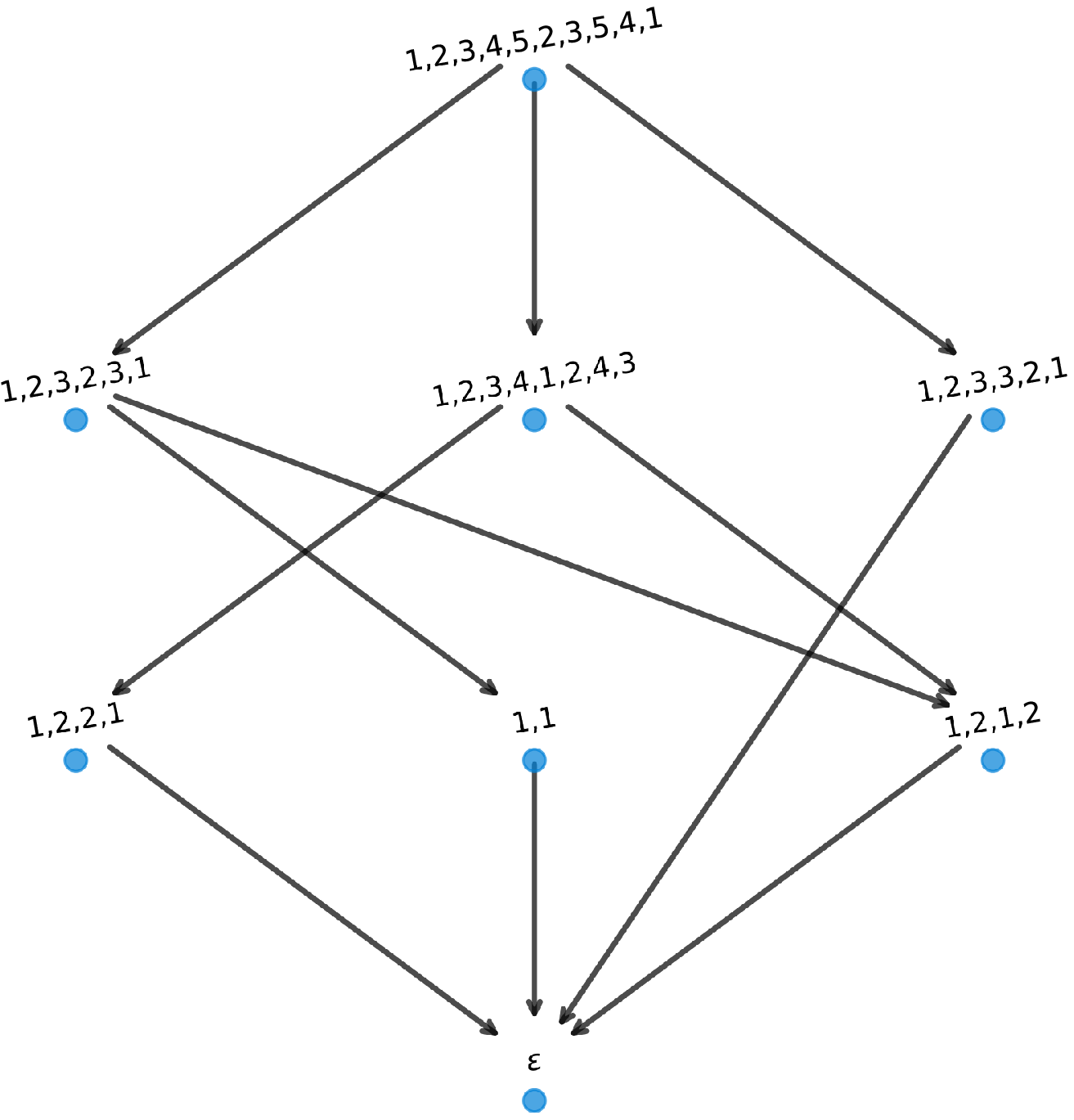}
\caption{Word graph rooted at $1234523541$.}
\label{fig:1234523541}
\end{figure}
\end{Ex}

\subsection{DOW Operations and Their Effect on Word Graphs}
\subsubsection{Operations that Result in Isomorphic Word Graphs}
\label{noeffect} 
We present some results on the cases where insertions, substitutions, or reversal do not affect the word graphs of DOWs. A prodsimplicial complex associated to $w\in \Sigma_{DOW}$ is the prodsimplicial complex associated with $G_w$. We consider operations that yield classes of DOWs whose complexes have similar topological properties. 

To find all predecessors of a given DOW $w$ we consider all possible insertions to $w$ up to ascending order equivalence. In \cite{digon} the insertions that yield equivalent DOWs and hence corresponding isomorphic word graphs were characterized. We present here results about DOWs that are not ascending order equivalent but yield isomorphic word graphs. 

\begin{Def}
Let $w, w' \in \DOW$ be such that $\Sigma[w] \cap \Sigma[w'] = \varnothing$. We define the {\em concatenation of $w$ and $w'$} as the DOW that is ascending order equivalent to $ww'$.
\end{Def}

\begin{Prop}
Let $w  \in \DOW$, and  $u\in \MSOW_w$. Let $v \in (\Sigma \setminus \Sigma[w])_{SOW}$ and let
\begin{align*}
w' &=xu_1v u_2y u_1 v u_2 z\\
w'' &=xu_1v u_2y u_2^R v^R u_1^R z
\end{align*}
where $u_1,u_2 \in \Sigma^*$ such that $u = u_1u_2$. Then $G_{w} \cong G_{w'}\cong G_{w''}$.
\end{Prop}

\begin{proof} 
In this case $u_1v u_2u_1v u_2$ is a maximal repeat word in $w'$. Note that $\MSOW_{w'} =( \MSOW_w \setminus \{u\} ) \cup \{u_1 v u_2\}$ but $d_{u}(w) = d_{u_1v u_2}(w')$. Hence, the word graphs of $w$ and $w'$ are isomorphic: $G_{w'} \cong G_w$. The case when $uu^R$ is a maximal return word is similar. 
\end{proof}

We call $w \in \DOW$ a {\em palindrome} if $w^R \sim w$. For example, the DOW $w = 123231$ is a palindrome, as $w^R = 132321 \sim 123231$. For a palindrome $w$, both $w$ and $w^R$ are in the same ascending order equivalence class so that the corresponding word graphs are the same. As a result, the reversal operation has no effect on the prodsimplicial complex obtained. 

\begin{Lem}
\label{lem:reverse}
Let $w \in \DOW$. Then $G_w \cong G_{w^R}$.
\end{Lem}

\begin{proof}
Note that $u$ is a repeat (resp. return) word in $w$ if and only if $u^R$ is a repeat (resp. return) word in $w^R$. Then if $v$ is a successor of $u$, for each edge $[u,v] \in G_w$ we have a corresponding edge $[u^R,v^R]\in E(G_{w^R})$. This bijection between the edges induces an isomorphism between the graphs.
\end{proof}

\begin{Ex}
Let $w = 122133$. Then $w^R = 331221 \sim 112332$. The word graphs rooted at $122133$ and $112332$ are isomorphic, as shown in \Cref{fig:revgraph}.

\begin{figure}[H]
\begin{center}
\includegraphics[width=0.3\textwidth]{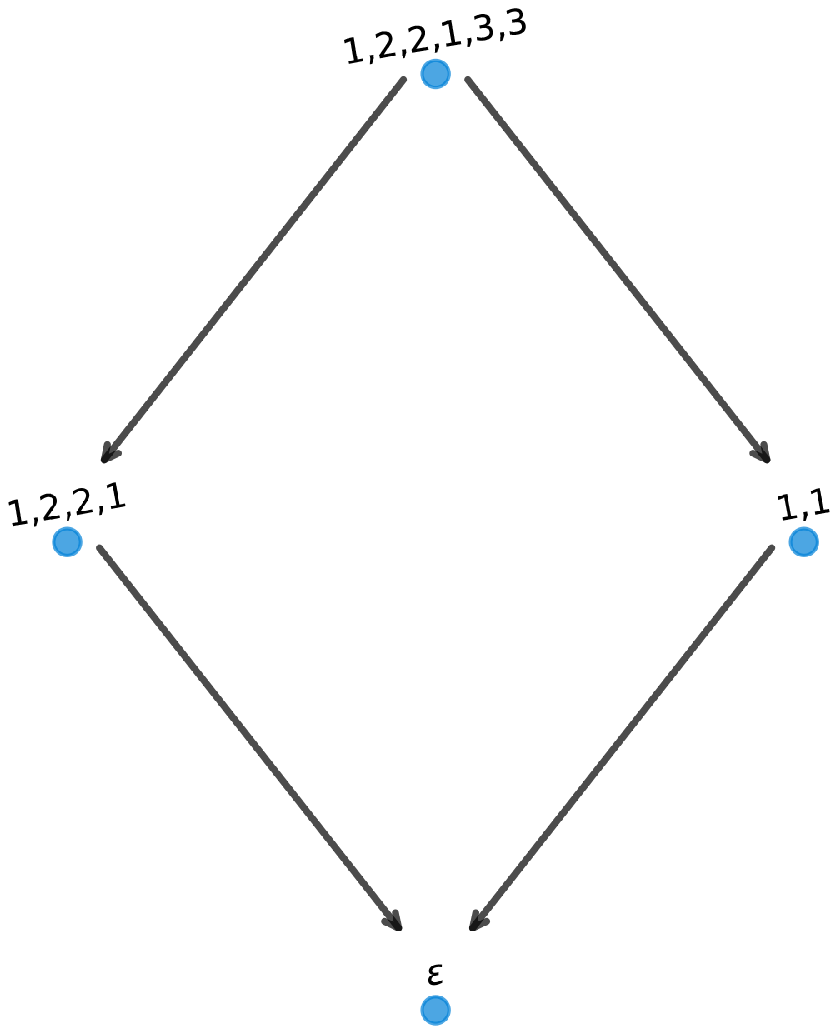}
\includegraphics[width=0.3\textwidth]{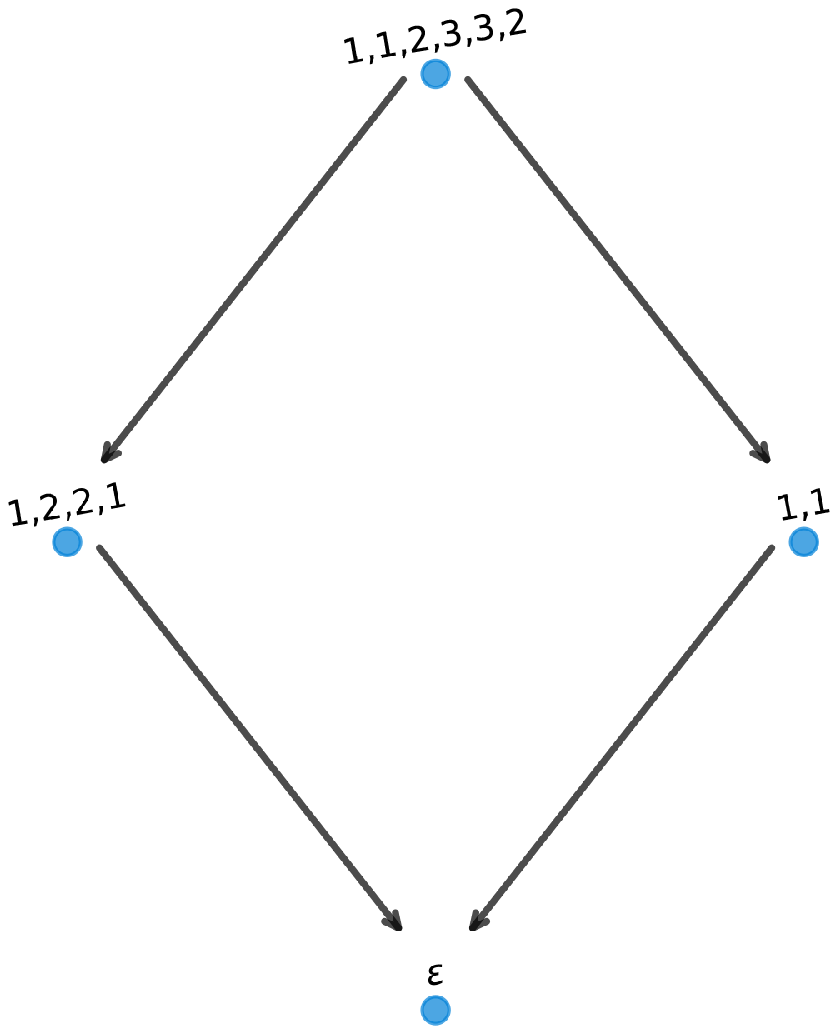}
\end{center}
\caption{The DOW $112332$ is ascending order equivalent to 
$122133^R$, so $G_{112332}$ is isomorphic to $G_{122133}$.}
\label{fig:revgraph}
\end{figure}
\end{Ex}

For the rest of this section, we consider the effect of substituting a repeat word $uu$ by the return word $uu^R$ on word graphs. In some cases we may substitute a maximal repeat word in a DOW $w$ for a maximal return word in $w$ without affecting the word graph. The following properties seem to play an important role in this context. 

\begin{Def}[square repeat or return words]
Let $w \in \DOW$ and let $u \in \MSOW_w$. We say that $u$ is a {\em square factor of $w$} if there exists $v \in \MSOW_w$ such that $v \neq u$ but $\abs{v}  = \abs{u}$.  A DOW $w$ is said to be {\em squarefree} if it has no square factors.
\end{Def}

Note that if $w$ is squarefree and $uu \in \MDOW_w$ then $vv, vv^R \not \in \MDOW_w$ for all $v\in \MSOW_w$ with $\abs{u}=\abs{v}$. Since no repeat or return factors appear more than once in $w$, we also have that all deletions $d_u(w)$ for $u \in \MSOW_w$ result in distinct DOWs. In particular, all subwords of $w$ are also squarefree.

\begin{Ex}[square factor, squarefree words]
\label{ex:squares}
Let $w = 12123434$. Then $\MDOW_w = \{12, 34\}$ and $12$ is a square repeat word in $w$. Let $x,y,z \in \SOW$ have distinct lengths. Then $w' = xyxzyz$ is squarefree. 
\end{Ex}

\begin{Def}[coprime]
Given two DOWs $w$ and $w'$ where $\Sigma[w] \cap \Sigma[w'] = \varnothing$, we say $w$ is {\em coprime to} $w'$ if all words of the form $uv$ where $u \in V(G_w)$ and $v \in V(G_{w'})$ are distinct in ascending order.
\end{Def}

By the definition for coprime words, words $w$ and $w'$ are coprime if for all $u,u' \in V(G_{w})$ and all $v,v' \in V(G_{w'})$, $u \neq u'$ and $v \neq v'$ implies $uv \not \sim u'v'$.

\begin{Ex}[coprime words]
\begin{sloppypar}
The word $w = 12234143$ has successors $\{1221, 112332, 123132, 11, \epsilon\}$, and the word $w' = 5678978956$ has successors $\{5656, 789789, \epsilon\}.$ It can be checked that all concatenations are distinct, therefore $w$ and $w'$ are coprime. 
\end{sloppypar}
\end{Ex}

\begin{Lem}
Let $w'$ be a successor of $w$. Then $G_{w'}$ is an induced subgraph of $G_w$.
\end{Lem}

\begin{Lem}
If $w$ is coprime to $w'$, then $w'$ is coprime to $w$. 
\end{Lem}
\begin{proof}
Suppose $w'$ is not coprime to $w$ and let $u,u'\in V(G_w)$ (where $u\neq u'$) and $v,v' \in V(G_{w'})$ (where $v\neq v'$) be such that $uv = u'v'$. Note that since $\epsilon \in V(G_w) \cap V(G_{w'})$, if there exists a nonempty DOW $x \in V(G_w) \cap V(G_{w'})$ then $w$ cannot be coprime to $w'$, as $\epsilon x = x\epsilon = x$. Without loss of generality, let $\abs{u}> \abs{u'}$ Then $u' = ux$ for some DOW $x$ (since $u \in \DOW$) and $v = xv'$, so that $\MSOW_v = \MSOW_{v'} \cup \MSOW_x$. Moreover, for all $y \in \MSOW_{v'}$, if $v_1=d_y(v)$ then $v = xv_1' $ where $v_1' = d_y(v')$. Inductively, if $v$ reduces to $v_1', v_2', \ldots, \epsilon$ via the deletion of $y_1, y_2, \ldots, y_k$, then $v$ reduces to $x$ via the deletion of $y_1, y_2, \ldots, y_k$. That is, $x \in V(G_w)$. Similarly, if $u' = ux$ we have that $x\in V(G_{w'})$, so that $w'$ is not coprime to $w$.
\end{proof}

Given this symmetry, instead of saying that $w$ is coprime to $w'$, we say that $w$ and $w'$ are coprime. 

In the following examples, $w=xyz$ and $u$ are coprime words and the word $w'$ (resp. $w''$) resulting from the insertion of $uu$ (resp. $uu^R$) in $w$ is squarefree. In one instance, the substitution results in $G_{w'}\cong G_w$, while in another it does not. We conjecture that squarefree and coprime  are necessary conditions for invariance 
of word graphs under substitution of repeat and return words.

\begin{Ex}[substitution results in isomorphic word graphs]
Let $x = 12$, $y = 345$, $z = 54312$ and $u = 6789$. Then $w' = xuyuz \sim 12 3456 789 345 987612$ and $w'' = xuyu^Rz \sim 12 345 6789 543 9876 12$. Here, $uu$ and $xyz$ are coprime, and both insertions result in squarefree words, with $G_{w'} \cong G_{w''}$ as depicted in \Cref{fig:notmxl1}.
\begin{figure}[htpb]
\centering
\includegraphics[width=0.46\textwidth]{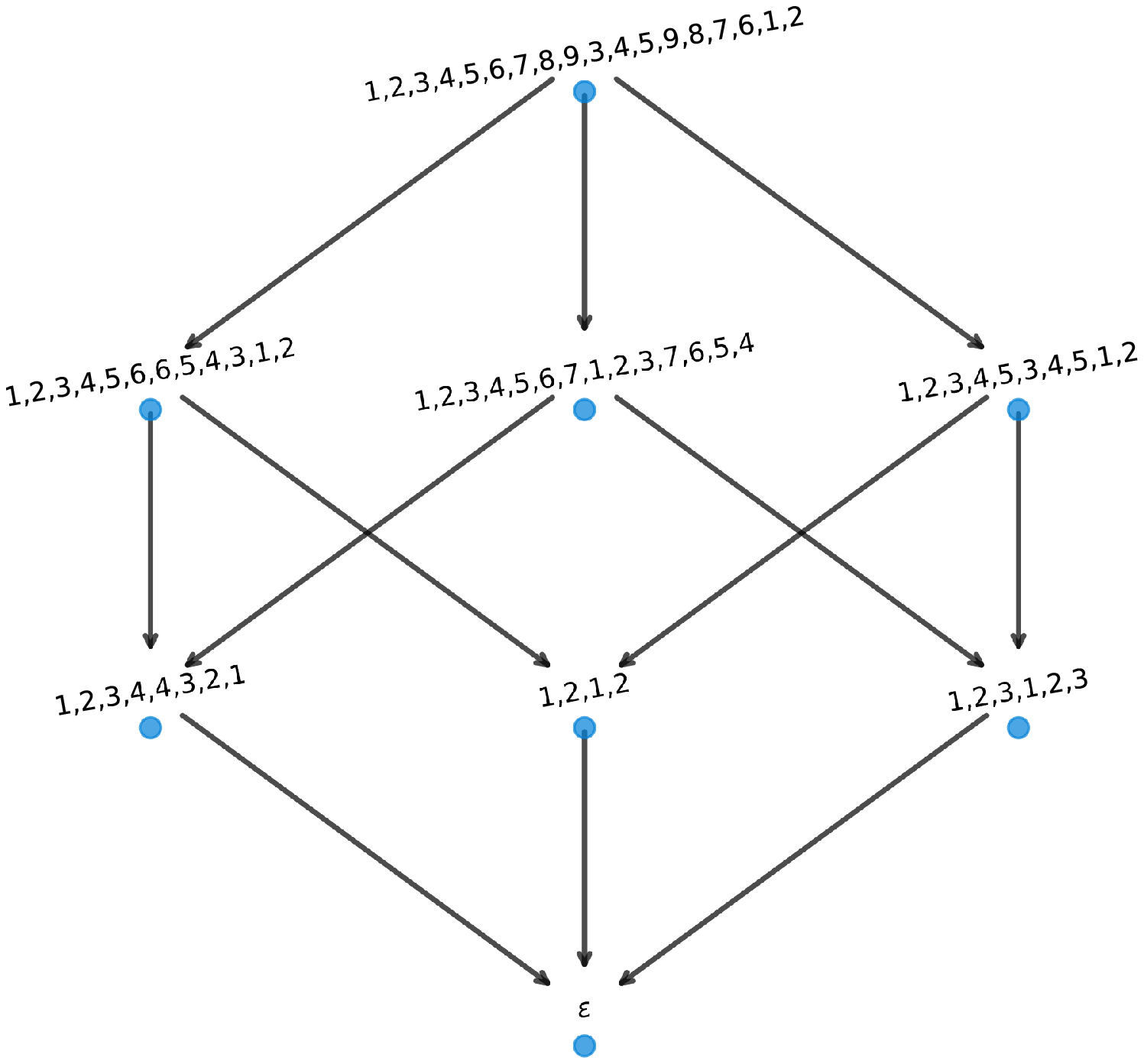}
\hspace{0.02\textwidth}
\includegraphics[width=0.46\textwidth]{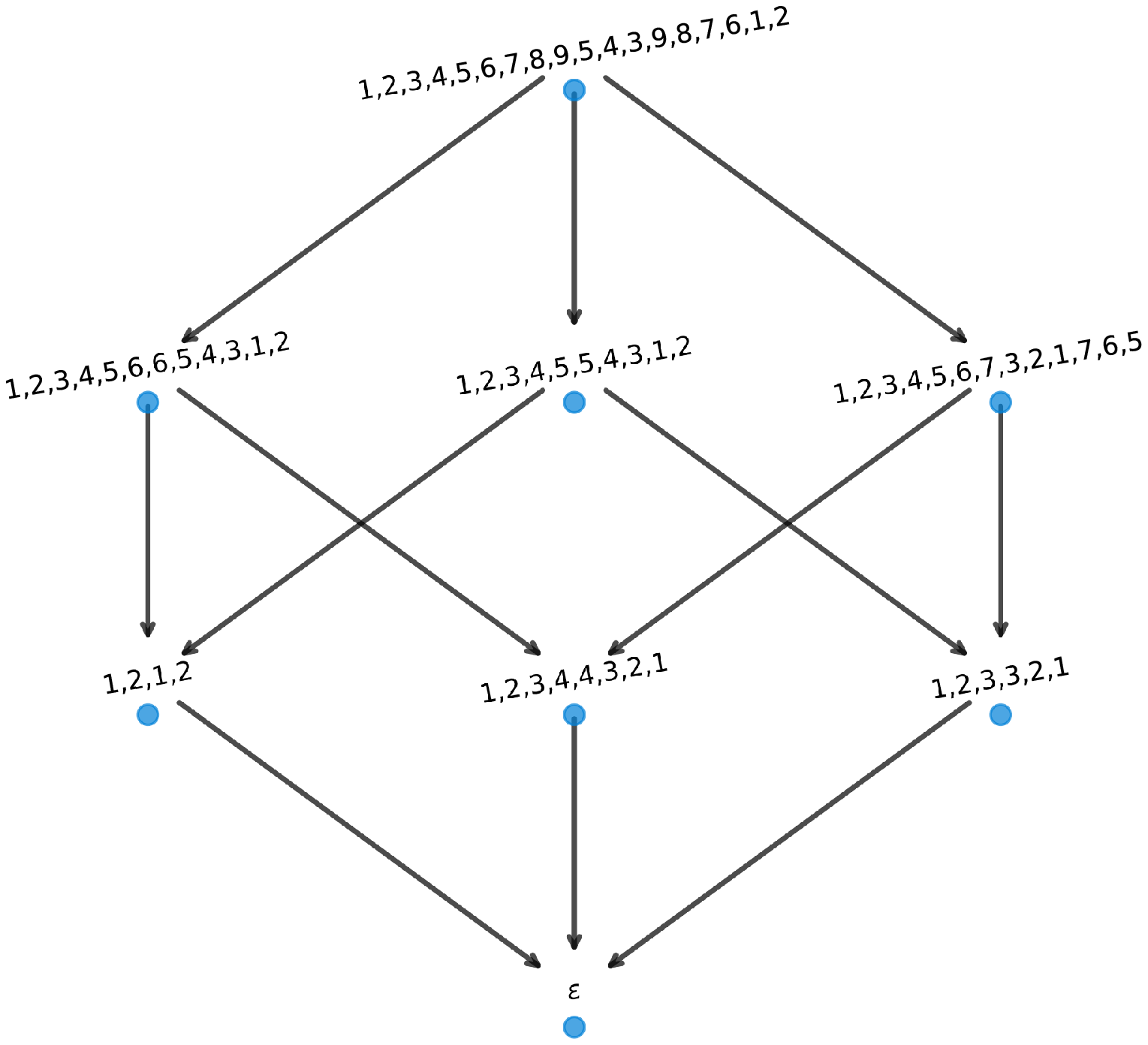}
\caption{The word graphs of $w'= xuyuz$ (left) and $w'' = xuyu^Rz$ (right) are isomorphic, as they both correspond to cubes.}
\label{fig:notmxl1}
\end{figure}
\end{Ex}

\begin{Ex}[substitution does not result in isomorphic word graphs]
Let $x = 1$, $y = 12345$, $z = 5432$ and $u = 67$. Then $w' = xuyuz \sim 1 23 1 4567 23 7654$ and $w'' = xuyu^Rz \sim 1 23 1 4567 32 7654$. Here we have that $uu$ is coprime to $xyz$ and the resulting insertions are squarefree but $G_{w'} \not \cong G_{w''}$ as depicted in in \Cref{fig:notmxml2}.

\begin{figure}[htpb]
\centering
\includegraphics[width=0.45\textwidth]{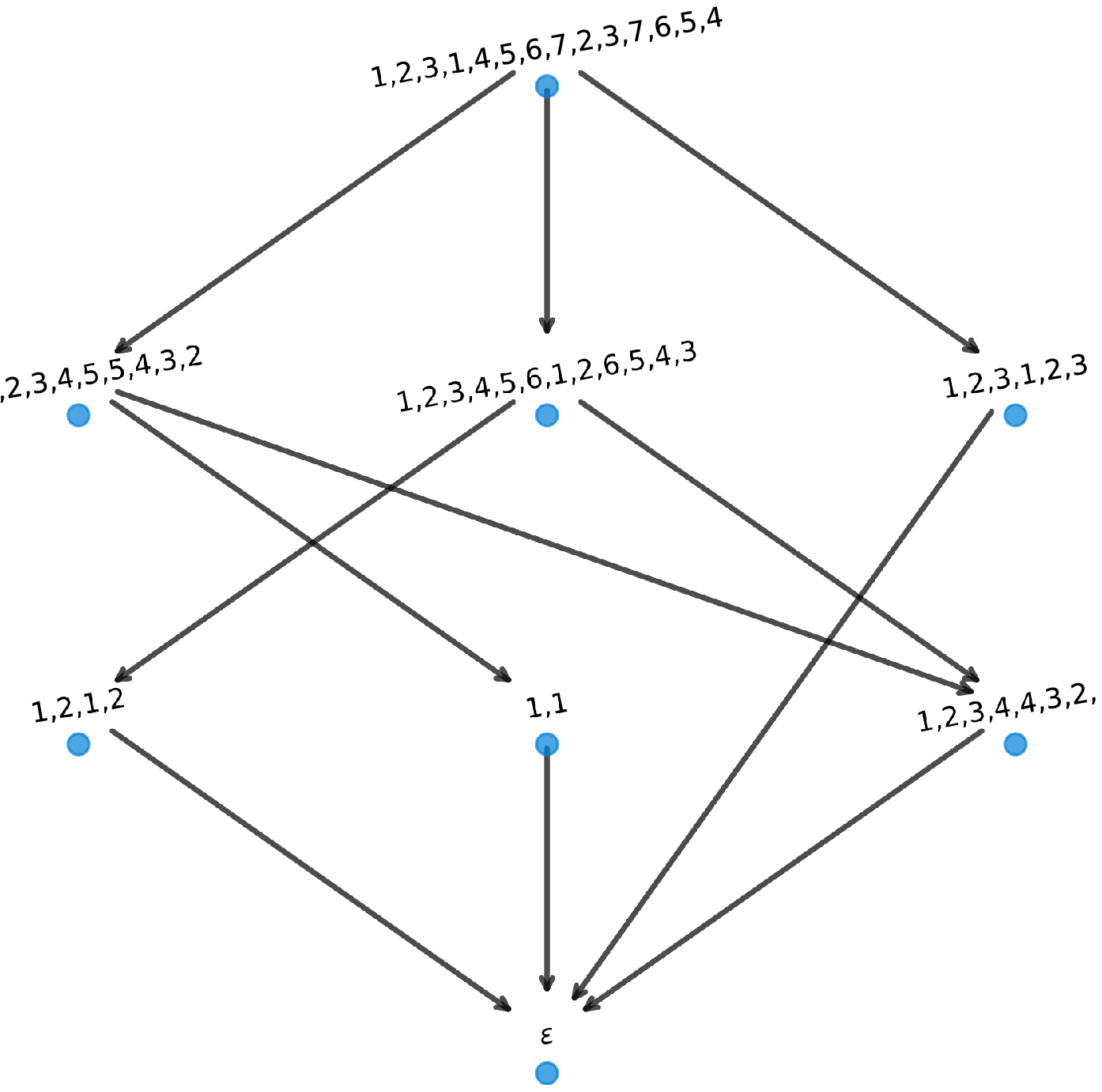}
\hspace{0.05\textwidth}
\includegraphics[width=0.45\textwidth]{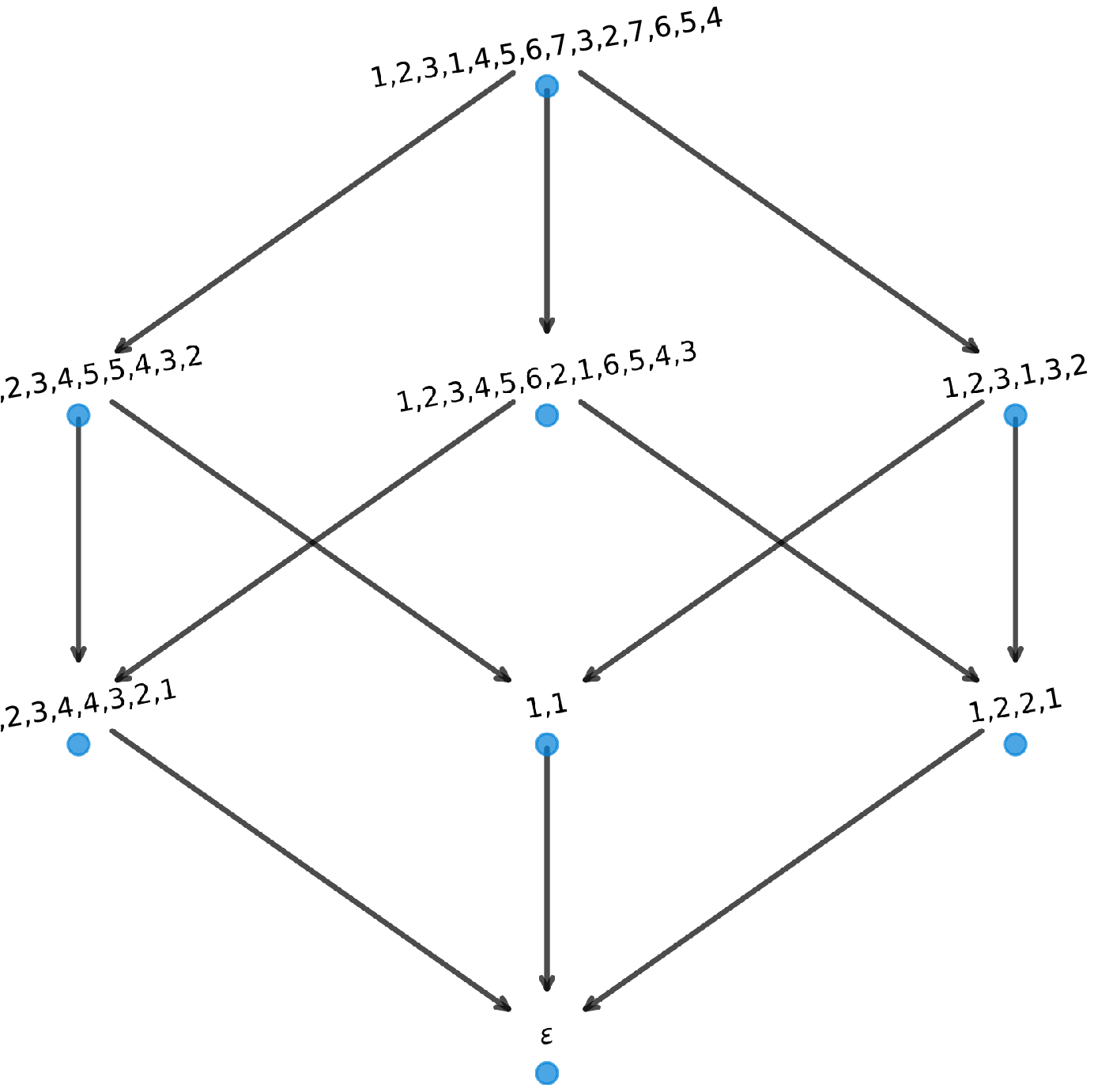}
\caption{The word graphs of $w' = xuyuz $ (left) and $w'' = xuyu^Rz$ (right) are not isomorphic.
}
\label{fig:notmxml2}
\end{figure}
\end{Ex}

\subsubsection{Doubling Effect}
Concatenation of a repeat (resp. return) word $uu$ (resp. $uu^R$) at the end of an existing word $w$ creates two subgraphs isomorphic to $G_w$ within $G_{wuu}$ (resp. $G_{wuu^R}$). These two subgraphs may not be disjoint unless we impose additional conditions, as discussed in the examples below.

\begin{Lem}
\label{lem:doubling}
Let $w\in \DOW$ and $uu \in (\Sigma \setminus \Sigma[w])_{DOW}$ be coprime. Then there exist two distinct (but possibly not disjoint) subgraphs $G_1$ and $G_2$ of $G_{wuu}$ isomorphic to $G_w$. Similarly, $G_{uuw}$ has two subgraphs isormophic to $G_w$. Moreover, we have  $G_{wuu} \cong G_{uuw}$.
\end{Lem}

\begin{proof}
Without loss of generality we prove the result for $G_{wuu}$ only. We claim that for each $v \in V(G_w)$, the word $vuu$ is ascending order equivalent to some $v' \in V(G_{wuu})$. Indeed, if $v$ is obtained from $w$ through iterated deletions so that 
\[v \sim d_{x_1}(d_{x_2}(d_{x_3}(\cdots d_{x_n}(w))\cdots))\]
where $x_i \not \sim u$ for all $1 \leq i \leq n$, then $v'\sim vuu$ is obtained from $wuu$ through
\[vuu \sim d_{x_1}(d_{x_2}(d_{x_3}(\cdots d_{x_n}(wuu))\cdots)).\]
Let $\iota: V(G_w) \rightarrow V(G_{wuu})$ be the inclusion of vertices and let $G_1$ be the graph induced by the image of $G_w$ in $G_{wuu}$. Let $G_2$ be the induced subgraph of $G_{wuu}$ generated by the set of vertices of the form $v' \sim vuu$ for $v \in V(G_w)$. Define $f: V(G_1) \rightarrow V(G_{2})$ where $v \mapsto v'\sim vuu$. Due to the above correspondence, $f$ is a bijection on the sets of vertices. Note that if $[v_1,v_2] \in G_w$ then $v_2 \sim d_x(v_1)$ for some $x \in \MSOW_{v_1}$. This holds if and only if $v_2uu \sim d_x(v_1uu)$ or $[v_1uu,v_2uu] \in G_{wuu}$ so that $f$ induces a bijection between edge sets that preserves vertex adjacency. It follows that $G_1 \cong G_2$.
\end{proof}

\begin{Ex}
The word graph corresponding to $12132344$ has two subgraphs isomorphic to the pentagon $G_{121323}$ in it as described in \Cref{lem:doubling}, which are not disjoint. This can be seen in \Cref{fig:12132344}, where the two subgraphs are highlighted in different colors.

\begin{figure}[h]
\begin{center}
\includegraphics[width=0.5\textwidth]{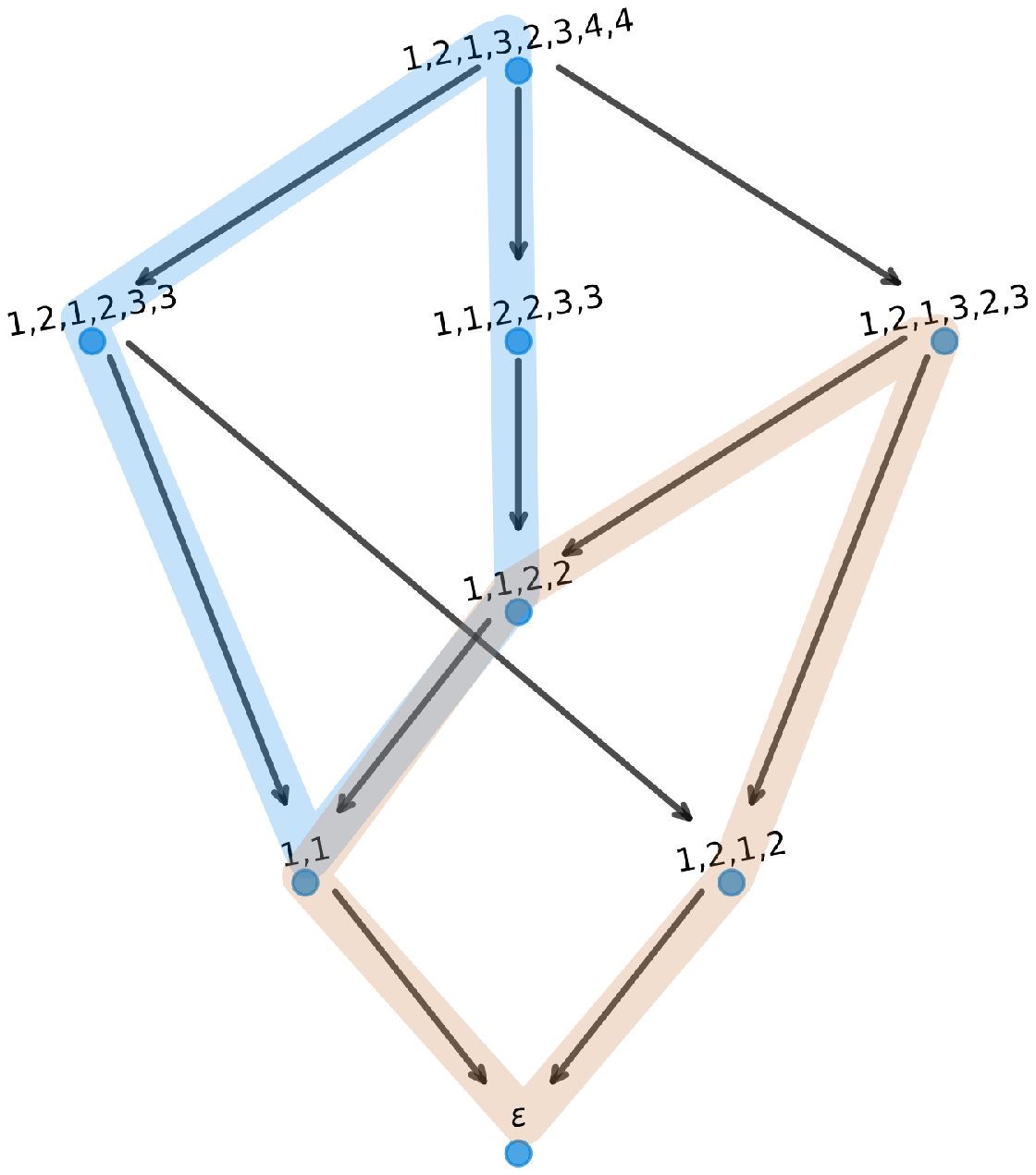}
\end{center}
\caption{Word graph of $12132344$ with two highlighted subgraphs isomorphic to $G_{121323}$.}
\label{fig:12132344}
\end{figure}
\end{Ex}

\subsubsection{Product Effect}
In this subsubsection we study word graphs of concatenated words.

\begin{Thm}
\label{thm:squarefree}
Let $w, w' \in \DOW $ be coprime. Then we have $G_{ww'} \cong G_{w} \square G_{w'}$.
\end{Thm}

\begin{proof}
We begin by noting that from the definition of concatenation, every vertex in $G_{ww'}$ can be decomposed into $u_iv_j$ where $u_i \sim u \in V(G_w)$ and $v_j \sim v \in V(G_{w'})$. We define a map $f: V(G_{w}\square G_{w'}) \rightarrow V(G_{ww'})$ where $(u,v)\mapsto u_iv_j$.
By hypothesis, since $w$ and $w'$ are coprime, $u \neq u'$ and $v \neq v'$ implies $uv \not \sim u'v'$ for all $u,u' \in V(G_{w})$ and all $v,v' \in V(G_{w'})$. This makes $f$ one-to-one. Surjectivity follows from the fact that no successors of $ww'$ will arise other than by concatenation of successors of $w$ and successors of $w'$. It follows that $f$ is a bijection on the sets of vertices.

We abuse the notation and write the vertices of $G_{ww'}$ as $uv$ where $\Sigma[u] \cap \Sigma[v] = \varnothing$, with $u \sim u'\in V(G_w)$ and $v\sim v' \in V(G_{w'})$ following the decomposition described above. Let $[(u_i,v_j), (u_i',v_j') ]\in E(G_{w}\square G_{w'})$. Then either $u_i=u_i'$ and $v_j' \sim d_x(v_j)$ for some $x \in \MSOW_{v_j}$ or $v_j = v_j'$ and $u_i' \sim d_y(u_i)$ for some $y \in \MSOW_{u_i}$. Note that since $w$ and $w'$ share no symbols we may write $\MDOW_{ww'} = \MDOW_w \sqcup \MDOW_{w'}$. We observe that for $u_i v_j, u_i' v_j' \in V(G_{ww'})$ there exists an edge $[u_iv_j,u_i'v_j'] \in E(G_{ww'})$ if and only if $u_i = u_i'$ and $[v_j, v_j'] \in E(G_{w'})$ or $[u_i, u_i'] \in E(G_w)$ and $v_j = v_j'$. This implies that $[(u_i,v_j), (u_i',v_j') ]\in E(G_{w}\square G_{w'})$ if and only if $[f((u_i,v_j)),f((u_i',v_j'))]\in E(G_{ww'})$, which indicates that $G_{ww'} \cong G_w \square G_{w'}$ as desired.
\end{proof}
In the case where $w'$ is a repeat or return word, with word graph having two vertices connected by an edge, we have the following corollary.
\begin{Rmk}
\label{prop:concatprod}
Note that if $w\in \DOW$ and $u\in \SOW$ are such that $uu \not \sim v$ for all $v \in V(G_w)$, then $G_{wuu} \cong G_w \square \Delta^1$. Similarly, if $uu^R \not \sim v$ for all $v \in V(G_w)$ then $G_{wuu^R} \cong G_w \square \Delta^1$.
\end{Rmk}

%% file: dowbettis.tex
\section{Betti Numbers and Generators for Word Graphs}
\label{dowbettis}
Having shown the result in general for arbitrary graphs, we now focus on consistently directed graphs corresponding to word graphs of DOWs to study the realizability of Betti numbers. Note that for DOWs $w$ and $w'$, if $w \in \MDOW_{w'}$ then $G_w\leq G_{w'}$. To find generators for the first and second homology groups among word graphs of DOWs we therefore turn to minimal size words whose corresponding graphs attain particular $\beta_1$ and $\beta_2$ values. 

A Python script (available at \href{https://github.com/fajardogomez/dowgraphs}{https://github.com/fajardogomez/dowgraphs}) was used to create the prodsimplicial complex associated with any directed graph $G$ and compute Betti numbers. 

\subsection{Betti Numbers of Word Graphs}
Let us consider the minimal size nontrivial 1-cycles among word graphs of DOWs. We claim that for the DOW $w$, $G_w$ satisfies $H_1(\Gamma(G_w)) \cong \Z$ and all other homology groups are trivial: 

$$1. \quad 121323, \hspace{15mm}
2. \quad 122331. $$

The corresponding word graphs are shown in \Cref{fig:pentagons}.

\begin{figure}[H]
\begin{center}
\includegraphics[width=0.3\textwidth]{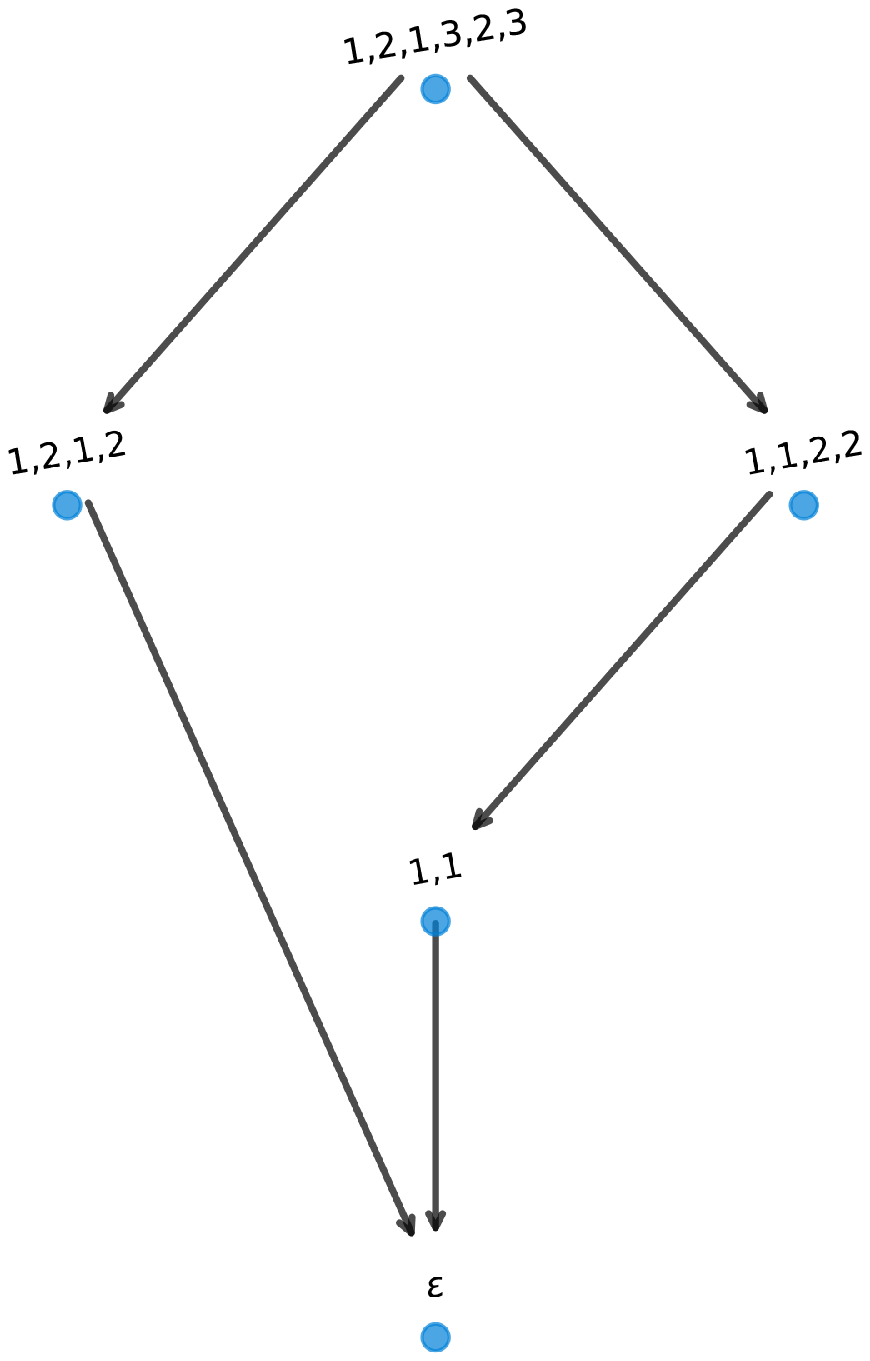}
\includegraphics[width=0.3\textwidth]{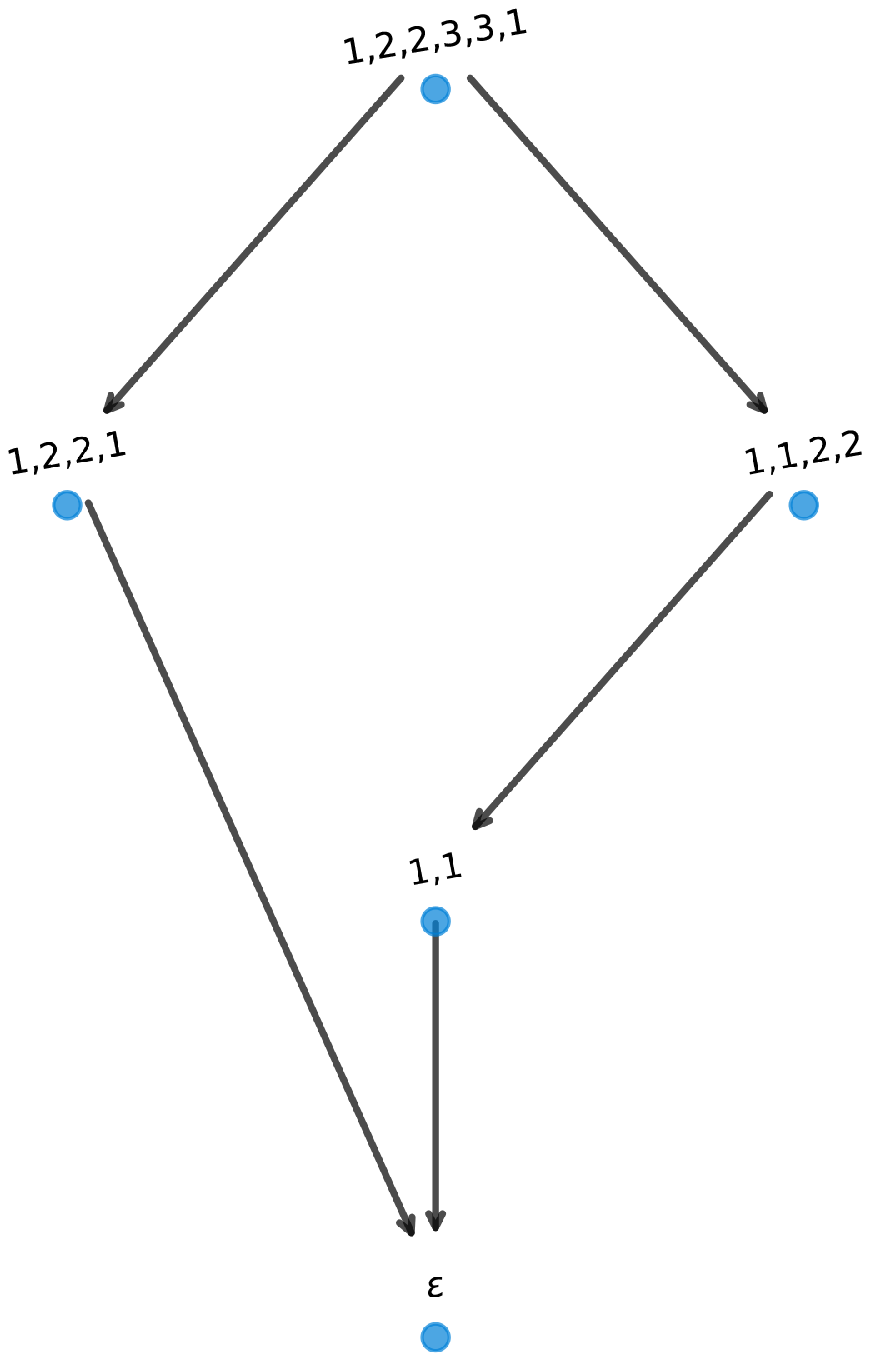}
\end{center}
\caption{The word graphs for DOWs $121323$ and $122331$ are isomorphic and contain nontrivial 1-cycles in their prodsimplicial complexes.}
\label{fig:pentagons}
\end{figure}

The result follows from \Cref{prop:cycles}, as the graphs consist of parallel paths from source to target of lengths 2 and 3, five edges in total, forming a pentagon. The DOWs on the least number of symbols that result in a generator for the first homology group are $121323$ and $122331$. 

Note that a $1$-cycle that is a square cannot be a generator of the first homology group because every square in a word graph (DOW graphs) consists of two paths of length 2, and therefore bounds a 2-cell. In all examples we have examined, $n$-gons representing nontrivial 1-cycles for $n\geq 6$ can be reduced to pentagons, so we conjecture that all nontrivial 1-cycles are homologous to pentagons. 
\begin{Ex}
Let $u_1$, $u_2$ and $\alpha$ be SOWs with $\Sigma[u_1] \cap \Sigma[u_2] \cap \Sigma[\alpha] = \varnothing$. The following word patterns result in pentagons and generalize the word graphs of $121323$ and $122331$:  $\alpha u_1u_1^R u_2 u_2^R\alpha^R$ and $u_1\alpha u_1 u_2\alpha u_2$ provided that $\abs{u_1} = \abs{u_2}$.
\end{Ex}

On the other side, all types of polyhedra that represent generators of nontrivial second homology discussed in \Cref{generalH2} can be realized in word graphs.

The complex corresponding to the word graph rooted at $12323414$ is homeomorphic to that of the graph in \Cref{lem:tennis}. This can be verified by referring to \Cref{fig:12234143}. Other word graphs have subgraphs isomorphic to those described in \Cref{lem:lantern} and \Cref{lem:banana}. For example, the word graph rooted at $12323144$ contains both. 

\begin{figure}[h]
\begin{center}
\includegraphics[width=0.3\textwidth]{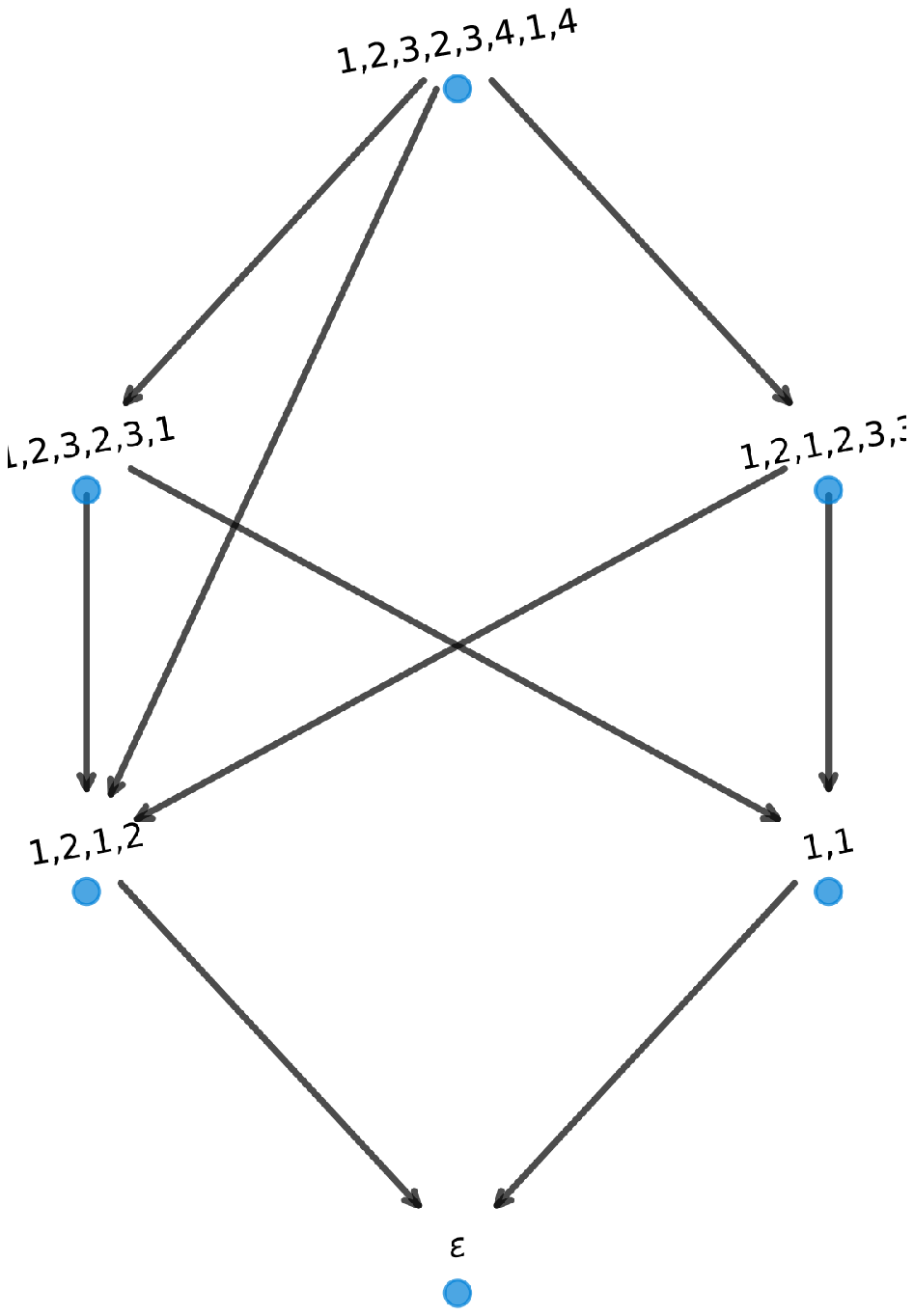}
\end{center}
\caption{Word graph rooted at $12323414$.}
\label{fig:12234143}
\end{figure}

\begin{center}
\begin{figure}[H]
\begin{tabular}[t]{lp{0.4\textwidth}lp{0.45\textwidth}}
(a) & \adjustbox{valign=t}{\includegraphics[width=0.45\textwidth]{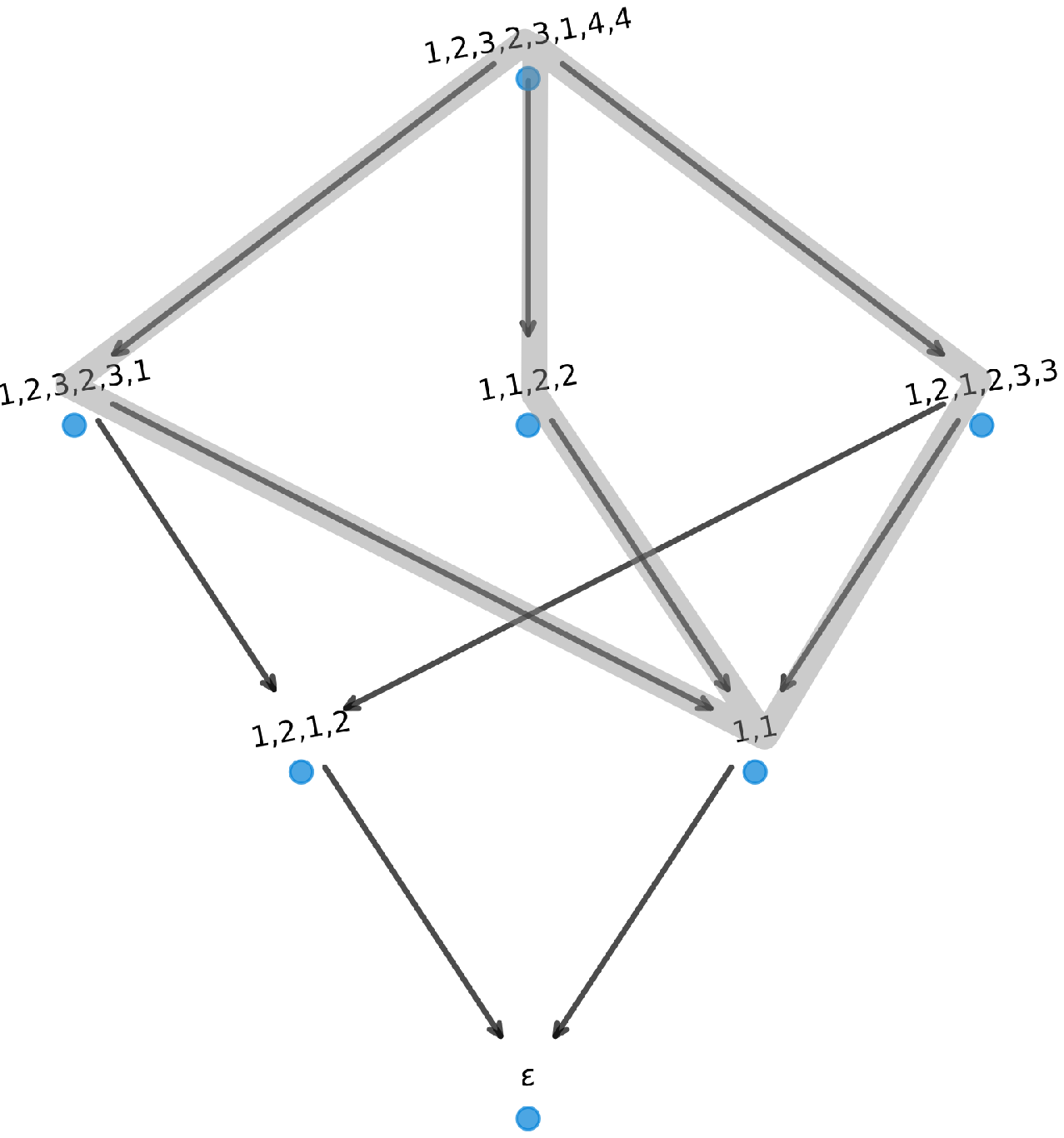}} &
(b) & \adjustbox{valign=t}{\includegraphics[width=0.4\textwidth]{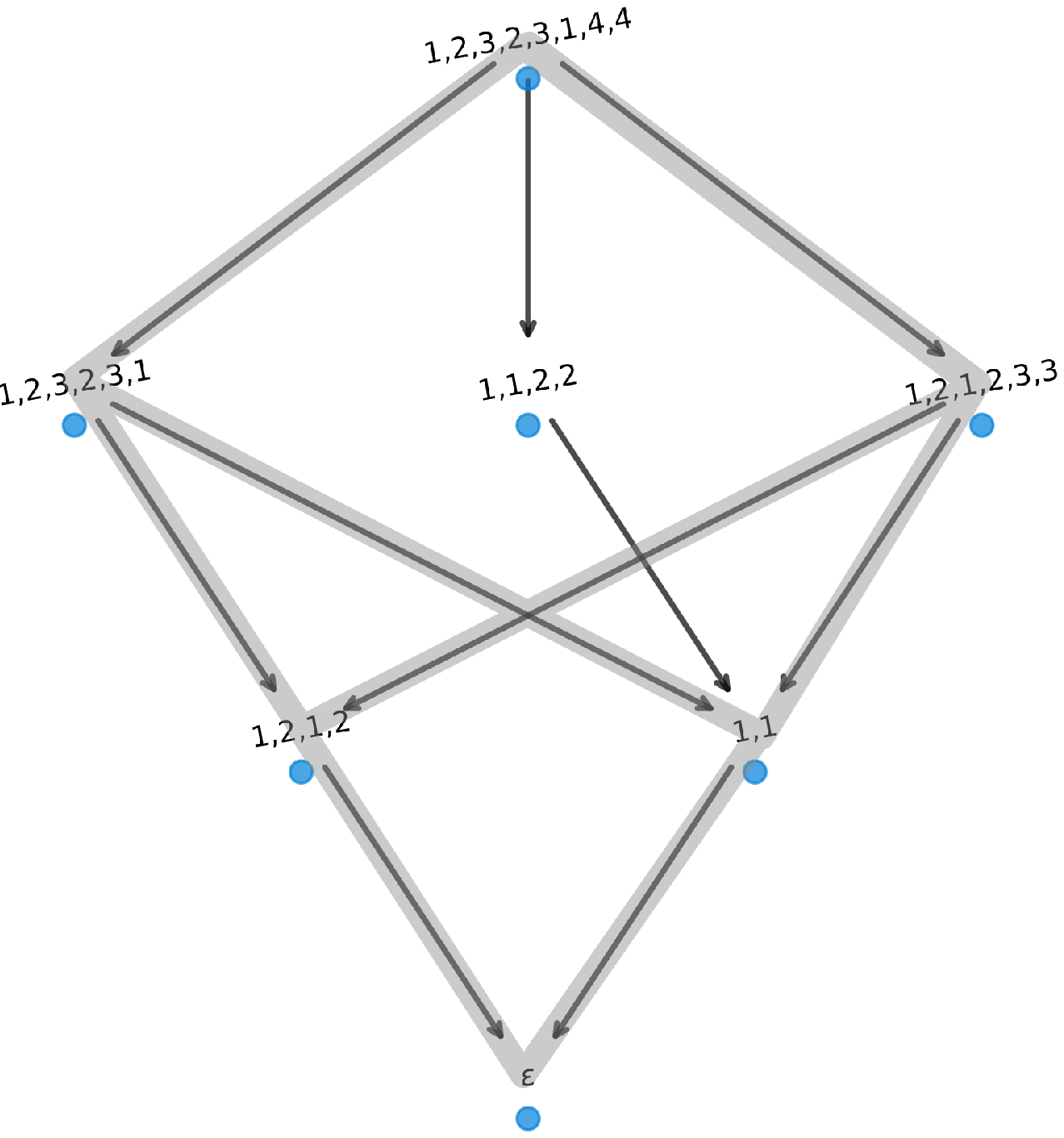}}
\end{tabular}
\caption{(a) Word graph rooted at 12323144 with a subgraph isomorphic to that in \Cref{fig:3squareslabeled} highlighted. (b) Word graph rooted at 12323144 with a subgraph isomorphic to that in \Cref{fig:tennisballs}(a) highlighted.}
\label{fig:12323144tennis}
\end{figure}
\end{center}

\subsection{Betti Numbers of Successors}
In general, if $u$ is a successor of $w$, it does not follow that $\beta_n(G_w) \geq \beta_n(G_u)$. Inspired by the work in \cite{assgraphs2}, we define the $n$-symbol {\em tangled cord} as the word
\[ t_n := 1213243\cdots (n-1)(n-2)n(n-1)n.\]

Note that $d_n(t_n)= t_{n-1}$ so that $t_{n-1} \in V(G_{t_n})$ for any $n>1$. For example, $t_6 = 121324354656$ has (the word corresponding to) the tangled cord of size five $t_5 = 1213243545$ as a successor. Based on results from the custom Python script, we have $\beta_1(G_{t_6}) = 1$ and $\beta_1(G_{t_5}) = 2$. This is the smallest pair of consecutive tangled cords where this inversion in Betti numbers is present. The next inversion occurs for $\beta_2$ when $n=11$, as shown in \Cref{tab:tangled}.

\begin{table}[h]
\centering
\begin{tabular}{lllll}
$n$ & $t_n$ & $\beta_1(G_{t_n})$ & $\beta_2(G_{t_n})$ & $\abs{V(G_{t_n})}$ \\ \hline
2 & 1,2,1,2 & 0 & 0 & 2 \\
3 & 1,2,1,3,2,3 & 1 & 0 & 5 \\
4 & 1,2,1,3,2,4,3,4 & 1 & 2 & 8 \\
5 & 1,2,1,3,2,4,3,5,4,5 & 2 & 6 & 13 \\
6 & 1,2,1,3,2,4,3,5,4,6,5,6 & 1 & 27 & 21 \\
7 & 1,2,1,3,2,4,3,5,4,6,5,7,6,7 & 1 & 54 & 34 \\
8 & 1,2,1,3,2,4,3,5,4,6,5,7,6,8,7,8 & 1 & 86 & 55 \\
9 & 1,2,1,3,2,4,3,5,4,6,5,7,6,8,7,9,8,9 & 1 & 111 & 89 \\
10 & 1,2,1,3,2,4,3,5,4,6,5,7,6,8,7,9,8,10,9,10 & 1 & 126 & 144 \\
11 & 1,2,1,3,2,4,3,5,4,6,5,7,6,8,7,9,8,10,9,11,10,11 & 1 & 116 & 233 \\
12 & 1,2,1,3,2,4,3,5,4,6,5,7,6,8,7,9,8,10,9,11,10,12,11,12 & 1 & 112 & 377
\end{tabular}
\caption{Tangled cords and invariants of their word graphs.}
\label{tab:tangled}
\end{table}

It is of interest to find  conditions for a successor $w'$ of $W$ such that $\beta_n(G_w) \geq \beta_n(G_{w'}) $ does not hold for some  $n \geq 1$.

%% file: conclusion.tex
\section{Concluding Remarks}
In this paper, a method to study topological properties of digraphs is proposed, by constructing a prodsimplicial complex for a given digraph. The construction provides an approach for TDA to be used on biological data that can be described with graphs. A specific such example, called the word graph, is presented that model assembly pathways in a ciliate species via DOWs \cite{dowdist}. The proposed construction of prodsimplicial complexes is applied to word graphs, and the Betti numbers are determined. The effect on word graphs under various operations, such as concatenation of words, are studied, and types of generators of homology were examined.

A number of problems remain unsolved. \Cref{prop:cycles} characterizes generators of the first homology groups for digraphs and \Cref{coro:digraphs} addresses the realizability of combinations of $\beta_1$ and $\beta_2$. However, we do not know all possible combinations of values of Betti numbers $\beta_1$ and $\beta_2$ over all word graphs. Though only a few specific examples of DOWs where the corresponding word graph has nontrivial 1-cycles and 2-cycles are included, more can be obtained through the DOW operations described in \Cref{noeffect}. For example, the concatenation of $121323$ with $4545$ corresponds to the Cartesian product of the word graphs (as shown in \Cref{prop:concatprod}) and therefore has no effect on Betti numbers so that $\beta_1(G_{1213234545}) = \beta_1(G_{121323}) = 1$. However, these methods would not create an exhaustive list of all graphs attaining particular given Betti numbers. It is desirable to characterize the effect of other word operations, such as insertions that disrupt maximal repeat or return words and thus may alter the homology groups of the complex on their word graphs.

We expect that the list of types of 1- and 2-cycles that represent nontrivial classes, such as the ones in \Cref{prop:cycles}, \Cref{lem:banana}, and \Cref{lem:tennis} may not be exhausted. We found pentagons as in \Cref{fig:pentagons} as nontrivial 1-cycles in word graphs. We do not know whether there are other polygons that represent nontrivial cycles. The \texttt{SageMath} \cite{sagemath} topology package has an option to find generators for the homology groups. However, the linear combinations in the output are not necessarily optimized to produce minimal generators. For instance, a generator for the first homology group in the word graph of $t_{10}$ is listed as a linear combination of nine edges. However, when these are drawn as an induced subgraph it becomes apparent that there exist edges between the vertices in the cycle so that it is homologous to a class represented by a pentagon.

It is also of interest to classify possible polyhedra that represent nontrivial second homology classes. 
We gave only partial answers in \Cref{digraphbettis}. Polyhedra such as ``pillows'' formed by two squares sharing all four boundary edges, are not allowed in the construction of the prodsimplicial complex. All word graphs computed here with nontrivial 1-cycles or 2-cycles in their corresponding prodsimplicial complex have subgraphs isomorphic to those in \Cref{badsquare,Lem:multiloops,prop:cycles,lem:banana,lem:lantern,lem:tennis}. The problem of finding an exhaustive list of generators using a minimal number of vertices and edges remains open. 

As a note on homology groups in general, \texttt{SageMath} outputs full groups instead of just Betti numbers as presented here. In an analysis of all DOWs of up to 7 symbols, no DOWs had a word graph whose corresponding prodsimplicial complex had torsion \cite{linathesis}. As the number of DOWs for each size increases superexponentially, homology groups were computed on longer DOWs for only samples of randomly generated DOWs and special word categories. As an interesting find, the homology groups of the tangled cord on ten symbols, $t_{10}$, had 2-torsion in the first (or second) homology group. Due to the large number of vertices and edges in the generator, however, it is not possible to find a type of generator of the 2-torsion. It remains an open problem to find and characterize DOWs whose corresponding prodsimplicial complexes have torsion in their homology groups.

The construction of prodsymlicial complexes for directed graphs presented in this paper is a first attempt to apply TDA to data sets consisting of graphs, designed specifically for the biological outputs called word graphs. Other graph outputs have been obtained in biology, and development of TDA for such outputs in a more general situations is desirable.
More generally, we propose constructions of custom-built cell complexes designed for the purpose of studying  individual biological problems. Polyhedral cells are to be  specified to build  a  complex, in such a way that are closed under boundary operators, and that are appropriate for a given biological situation.

Word graphs of DOWs correspond to reductions of chord diagrams, and we expect applications to areas related to chord diagrams. The reductions of repeat and return words can be generalized to other rewriting systems.It is desirable to apply similar constructions of complexes and use of their homology for confluent rewriting systems.

\subsection*{Acknowledgments}
This research was (partially) supported by the grants NSF DMS-2054321, CCF-2107267, The Simon’s Fellow grant from the Simons Foundation, the W.M. Keck Foundation. In addition this research was under auspices of the Southeast Center for Mathematics and Biology, an NSF-Simons Research Center for Mathematics of Complex Biological Systems, under National Science Foundation Grant No. DMS-1764406 and Simons Foundation Grant No. 594594.

%% file: biblio.bib
@article{digon,
    title="Insertions Yielding Equivalent Double Occurrence Words",
    author="Daniel A. Cruz and Margherita Maria Ferrari and Nata\v sa Jonoska and Lukas Nabergall and Masahico Saito",
    year="2020",
    journal="Fundamenta Informaticae",
    volume="171",
    pages="113-132"
}

@article{CIST,
  title={Homology for quandles with partial group operations},
  author={Carter, J. Scott and Ishii, Atsushi and
Saito, Masahico
and Tanaka, Kokoro},
  year={2017},
  journal={Pacific Journal of Mathematics},
  volume=287,
  number=1,
  pages={19-48}
}

@article{CLY,
  title={A prismatic classifying space},
  author={Carter, J.~Scott and Lebed, Victoria
  and Yang, Seung Yeop},
  BOOKTITLE = {Nonassociative mathematics and its applications},
    SERIES = {Contemp. Math.},
    VOLUME = {721},
     PAGES = {43--68},
 PUBLISHER = {Amer. Math. Soc., Providence, RI},
      YEAR = {2019}
}

@book{hatcher,
  title={Algebraic Topology},
  author={Hatcher, Allen },
  series={Algebraic Topology},
  year={2002},
  address={New York},
  publisher={Cambridge University Press}
}

@article{vizing,
  title={The cartesian product of graphs. },
  author= { Vizing, Vadim Georgievich},
  year={1963},
  journal={Vy\v{c}isl. Sistemy},
  volume=9,
  pages={30-43}
}

@article{ryan,
author = {Arredondo, Ryan},
address = {Ithaca},
copyright = {2013. This work is published under http://arxiv.org/licenses/nonexclusive-distrib/1.0/ (the “License”). Notwithstanding the ProQuest Terms and Conditions, you may use this content in accordance with the terms of the License.},
issn = {2331-8422},
journal = {arXiv.org},
language = {eng},
publisher = {Cornell University Library, arXiv.org},
title = {Reductions on Double Occurrence Words},
year = {2013},
}

@book{diestel,
  author = { Diestel, Reinhard},
  publisher = {Springer},
  address = {New York},
  title = {Graph Theory (Graduate Texts in Mathematics)},
  year = {2005}
}

@book{imrich,
  author = {Imrich, Wilfried and Klav\v{z}ar, Sandi},
  isbn = {0471370398},
  publisher = {Wiley},
  title = {Product Graphs: Structure and Recognition},
  address={New York},
  year = 2000
}

@article{feigenbaum,
title = "Directed cartesian-product graphs have unique factorizations that can be computed in polynomial time",
journal = "Discrete Applied Mathematics",
volume = "15",
number = "1",
pages = "105 - 110",
year = "1986",
author = "Joan Feigenbaum"
}

@article{patterns,
title = "Recurring patterns among scrambled genes in the encrypted genome of the ciliate Oxytricha trifallax",
journal = "Journal of Theoretical Biology",
volume = "410",
pages = "171 - 180",
year = "2016",
author = "Jonathan Burns and Denys Kukushkin and Xiao Chen and Laura F. Landweber and Masahico Saito and Nataša Jonoska"
}

@book{kozlov,
  title={Combinatorial Algebraic Topology},
  author={Kozlov, Dimitry},
  volume={21},
  year={2007},
  address={New York},
  publisher={Springer Science \& Business Media}
}

@article{ppaths,
  title={Path Complexes and their Homologies},
  author={Alexander A. Grigor'yan and Yu V. Muranov and Yong Lin and Shing-Tung Yau},
  year={2020},
  journal={Journal of Mathematical Science},
  volume={248},
  pages={564-599}
}

@misc{preprint,
      title={Homologies of path complexes and digraphs}, 
      author={Alexander Grigor'yan and Yong Lin and Yuri Muranov and Shing-Tung Yau},
      year={2013},
      eprint={1207.2834},
      archivePrefix={arXiv},
      primaryClass={math.CO}
}

@article{hajij,
title = {Graph based analysis for gene segment organization In a scrambled genome},
journal = {Journal of Theoretical Biology},
volume = {494},
pages = {110215},
year = {2020},
issn = {0022-5193},
author = {Mustafa Hajij and Nataša Jonoska and Denys Kukushkin and Masahico Saito},
}

@article{assgraphs2,
title = {Four-regular graphs with rigid vertices associated to DNA recombination},
journal = {Discrete Applied Mathematics},
volume = {161},
number = {10},
pages = {1378-1394},
year = {2013},
issn = {0166-218X},
author = {Jonathan Burns and Egor Dolzhenko and Nataša Jonoska and Tilahun Muche and Masahico Saito}
}

@article{dowdist,
author = {Jonoska, Nata\v{s}a and Nabergall, Lukas and Saito, Masahico},
year = {2017},
month = {08},
pages = {225-238},
title = {Patterns and Distances in Words Related to DNA Rearrangement},
volume = {154},
journal = {Fundamenta Informaticae}
}

@misc{sagemath,
  author={The Sage Developers and William Stein and David Joyner and David Kohel and John Cremona and Eröcal, Burçin},
  title={SageMath, version 9.0},
  year={2020},
  url={http://www.sagemath.org},
}

@misc{ryanthesis,
copyright = {Database copyright ProQuest LLC; ProQuest does not claim copyright in the individual underlying works.},
isbn = {130384690X},
language = {eng},
publisher = {ProQuest Dissertations Publishing},
title = {Properties of graphs used to model DNA recombination},
year = {2014},
author = {Arredondo, Ryan C.},
}

@article{directedcliques,
issn = {1662-5188},
journal = {Frontiers in computational neuroscience},
language = {eng},
pages = {48-48},
publisher = {Frontiers Research Foundation},
title = {Cliques of Neurons Bound into Cavities Provide a Missing Link between Structure and Function},
volume = {11},
year = {2017},
author = {Reimann, Michael W and Nolte, Max and Scolamiero, Martina and Turner, Katharine and Perin, Rodrigo and Chindemi, Giuseppe and Dłotko, Paweł and Levi, Ran and Hess, Kathryn and Markram, Henry},
address = {Switzerland},
}

@article{cliquecomplex,
copyright = {Masulli and Villa. 2016},
issn = {2193-1801},
journal = {SpringerPlus},
language = {eng},
number = {1},
pages = {388-388},
publisher = {Springer International Publishing},
title = {The topology of the directed clique complex as a network invariant},
volume = {5},
year = {2016},
author = {Masulli, Paolo and Villa, Alessandro E. P},
address = {Cham},
}

@book{digraphclasses,
edition = {1st ed. 2018.},
isbn = {3-319-71840-1},
language = {eng},
publisher = {Springer International Publishing},
series = {Springer Monographs in Mathematics},
title = {Classes of Directed Graphs},
year = {2018},
address = {Cham},
}

@phdthesis{linathesis,
author = "Fajardo G\'omez, Lina",
year = "2022",
title = "Methods in Discrete Mathematics to Study DNA Rearrangement Processes",
school = "University of South Florida"
}

@incollection{landweberref,
issn = {1619-7127},
language = {eng},
pages = {349-359},
publisher = {Springer Berlin Heidelberg},
series = {Natural Computing Series},
title = {Insights into a Biological Computer: Detangling Scrambled Genes in Ciliates},
author = {Cavalcanti, Andre R. O. and Landweber, Laura F.},
address = {Berlin, Heidelberg},
booktitle = {Nanotechnology: Science and Computation},
copyright = {Springer-Verlag Berlin Heidelberg 2006},
isbn = {3540302956},
}

@article{prestonref,
issn = {0192-253X},
journal = {Developmental genetics},
number = {1},
pages = {66-74},
publisher = {Wiley Subscription Services, Inc., A Wiley Company},
title = {Scrambled actin I gene in the micronucleus of Oxytricha nova},
volume = {13},
year = {1992},
author = {Prescott, David M. and Greslin, Arthur F.},
address = {Hoboken},
copyright = {Copyright © 1992 Wiley‐Liss, Inc.},
}

@article{landweberref2,
journal = {Microbiology spectrum},
number = {6},
title = {Programmed Genome Rearrangements in the Ciliate Oxytricha},
volume = {2},
year = {2014},
author = {Yerlici, V Talya and Landweber, Laura F},
address = {United States},
}
